\documentclass[10pt,a4paper]{article}
\usepackage[utf8]{inputenc}
\usepackage[english]{babel}
\usepackage{amsmath}
\usepackage{amsthm}
\usepackage{amsfonts}
\usepackage{amssymb}
\usepackage{graphicx}
\usepackage[left=2cm,right=2cm,top=2cm,bottom=2cm]{geometry}

\usepackage{booktabs}
\usepackage{array}

\usepackage{hyperref}
\usepackage{booktabs}

\numberwithin{equation}{section}

\author{Fausto Colantoni\footnote{Ulm University, fausto.colantoni@uni-ulm.de} \and Mirko D'Ovidio\footnote{Sapienza University of Rome, mirko.dovidio@uniroma1.it}}

\newcommand{\indep}{\perp \! \! \! \perp}

\theoremstyle{plain}
\newtheorem{theorem}{Theorem}
\newtheorem{lemma}{Lemma}
\newtheorem{corollary}{Corollary}

\theoremstyle{remark}
\newtheorem{remark}{Remark}

\begin{document}
\title{On the telegrapher's signals of sticky local times}

\maketitle

\begin{abstract}
We consider the boundary trace process of a new class of sticky Brownian motions and  study the telegraph signals of the boundary local time. For the fractional telegraph equation
\begin{align*}
(\sigma/\eta) D^\alpha_t v(t,x)  + D^{2\alpha}_t v(t,x) = \frac{\partial^2 v}{\partial x^2}(t,x), \quad t>0,\, x \in \mathbb{R}, \quad \alpha \in (0,1)
\end{align*}
we provide a probabilistic representation of the solution in anisotropic Sobolev spaces and compare it with well-known representations in the literature. We subsequently discuss the associated processes and provide their pathwise representations. Based on these representations, we introduce a characterization of the local times for a wide class of sticky Brownian motions on $\Omega$ governed by the non-local dynamic boundary condition
\begin{align*}
\alpha \in (0,1], \qquad \eta D^\alpha_t \varpi(t,x) = - \sigma \partial_{\bf n} \varpi(t,x), \qquad t>0, \; x \in \partial \Omega
\end{align*}
where $D^\alpha_t$ denotes the fractional derivative in the Caputo-D\v{z}rba\v{s}jan sense. We focus mainly on the interval $[a,b]$ to construct a prototype model, and then lay the foundation for the analysis on balls of $\mathbb{R}^d$. The fractional telegraph equation smoothly interpolates between wave propagation and diffusion under anomalous dynamic. Stochastically, this behaviour is driven by the underlying sticky Brownian motion experiencing significantly prolonged trapping times on the boundary. The solution $v$ retains its continuity up to $t=0$ in $H^1(\mathbb{R})$, which corresponds to the mean-square continuity of the telegrapher's process at the initial instant. However, for $t>0$, the severe sticky effect causes the stochastic trajectories to undergo prolonged trapping periods leading to highly irregular and rough paths for the local time of the sticky Brownian motion. We remark that in our construction the Brownian structure appears immediately, rather than only as a hydrodynamic limit.
\end{abstract}

\tableofcontents

\newpage

\section{Introduction}

Fractional-order differential equations have become an essential tool for modelling non-local and memory-dependent evolutionary processes. A prime representative of such non-local formulations is the fractional telegraph equation, which arises as a fundamental model governing finite-velocity random flights, non-Markovian diffusion processes, and wave propagation in lossy, dispersive media. Motivated by the profound importance of this equation, this study aims to provide a probabilistic representation of the solution, valid for any fractional order $2\alpha \in (0,2]$, which holds in a weak sense. Crucially, the underlying path representations exhibit either a linear or a staircase behaviour that cannot be recovered in the weak sense alone. These two distinct path behaviours admit an interesting interpretation when associated with the local times of Brownian motion. For example, the linear path behaviour can be associated with a local time measured on the local time scale.

\subsection{Preliminary settings}
We consider the following setting:
\begin{itemize}
\item[-] $\{\theta_i\}_{i \in \mathbb{N}_0}$ is a sequence of i.i.d. random variables with
\begin{align*}
\mathbf{P}(\theta_0 = -1) = \mathbf{P}(\theta_0=+1) = \frac{1}{2};
\end{align*}
\item[-] $L=\{L_t\}_{t\geq 0}$ with $L_t = \inf\{s\geq 0\,:\, H_s >t\}$ is the inverse process of a stable subordinator $H=\{H_t\}_{t\geq 0}$;
\item[-] $N=\{N_t\}_{t \geq 0}$ is a Poisson process with $\mathbf{P}(N_t \geq j) = \mathbf{P}(T_j < t)$ and $\{T_j\}_{j \in \mathbb{N}_0}$ where
$$T_j - T_{j-1} =: e_j, \quad T_{-1}=0$$ 
is the sequence of independent exponential inter times;
\item[-] $N^L = \{N^L_t\}_{t\geq 0}$ is the renewal process obtained via time change as $N^L_t=N_{L_t}$. Define the sequence $\{T^L_j\}_{j \in \mathbb{N}_0}$ where
$$T^L_j - T^L_{j-1} =: e^L_j, \quad T^L_{-1}=0.$$ 
By exploiting the Markov nature of $H$, we know that (see \cite[Theorem 6.8 and Theorem 6.15]{EJP}) 
\begin{align*}
\textrm{$\{e^L_j\}_{j \in \mathbb{N}_0}$ is a sequence of i.i.d. random variables. In particular, $e^L_j \stackrel{d}{=} H_{e_j}$ for all $j$.}
\end{align*}
For simplicity, we proceed under the assumption $e_0=0$ and, for $j \in \mathbb{N}$, 
$$\mathbf{P}(e_j >t) = e^{- (\sigma/\eta) t}, \quad t \geq 0 \quad \textrm{with constants} \quad \eta>0,\, \sigma >0.$$
In particular, $e_0=0$ implies $T_0 = T^L_0=0$, $e^L_0=0$ and, for $j \in \mathbb{N}$,
\begin{align*}
\mathbf{P}(e^L_j > t) = E_\alpha(-(\sigma/\eta) t^\alpha), \quad t \geq 0
\end{align*}
is written in terms of the Mittag-Leffler function $E_\alpha(\cdot)$;
\item[-] We assume
\begin{align}
\textrm{$N \indep H$ and $\{\theta_i\}_{i \in \mathbb{N}} \indep \{e_i\}_{i \in \mathbb{N}}$.}
\label{AllIndep}
\end{align}
\item[-] Moreover, $X^+ =\{X^+_t\}_{t \geq 0}$ denotes a reflected Brownian motion on a domain with boundary local time $\gamma^+=\{\gamma^+_t\}_{t\geq 0}$. We denote by $\mathcal{F}^+_t = \sigma\{X^+_s, \, 0\leq s < t\}$ the natural filtration and write 
\begin{align*}
\mathbf{E}[f(X^+_s)| \mathcal{F}^+_t] = \mathbf{E}[f(X^+_s)|X^+_t], \quad t \leq s,
\end{align*}
for a function $f$ sufficiently regular.
\end{itemize}

\subsection{Fractional derivatives in a nutshell}
 
The non-local operator in time $D^\alpha_t$ is termed Caputo-D\v{z}rba\v{s}jan derivative of order $\alpha$ and it is defined as 
\begin{align}
D^\alpha_t \varphi(t,x) :=
\left\lbrace
\begin{array}{ll}
\displaystyle  \frac{1}{\Gamma(n-\alpha)} \int_0^t \frac{\partial^n \varphi}{\partial s^n}(s,x) (t-s)^{n-\alpha-1} ds, & n = \lceil \alpha \rceil,\; \alpha \in (0,1) \cup (1,2), \\
\\
\displaystyle \frac{\partial \varphi}{\partial t}(t,x), & \alpha=1,\\
\\
\displaystyle \frac{\partial^2 \varphi}{\partial t^2}(t,x), & \alpha=2.
\end{array}
\right.
\label{CDderDef}
\end{align}  
For $z>0$, $\Gamma(z) = \int_0^\infty s^{z-1} e^{-s}ds$ is the absolutely convergent Euler integral. As usual, we refer to $\Gamma(\cdot)$ as the gamma function. The Laplace transform 
\begin{align}
\int_0^\infty e^{-\lambda t} D_t^{\alpha} \varphi\, dt = \lambda^{\alpha} \int_0^\infty e^{-\lambda t} \varphi\, dt - \sum_{k=0}^{n-1} \lambda^{\alpha -1-k}\varphi^{(k)}(0), \quad \lambda>0
\label{LAPder}
\end{align}
is a powerful tool for analysing the operator $D^\alpha_t$. The Caputo-D\v{z}rba\v{s}jan fractional derivative has been introduced in \cite{caputoBook, CapMai71, CapMai71b} by Caputo (together with Mainardi) and separately in a series of works starting from \cite{Dzh66, DzhNers68} by D\v{z}rba\v{s}jan.

\subsection{A brief review on the integrated telegraph process}
\label{sec:introTel}

Since the seminal works of Goldstein \cite{Goldstein51} and Kac \cite{Kac74}, the telegraph process has served as the fundamental stochastic counterpart to the classical telegraph equation, effectively modeling transport phenomena with finite velocity. The telegraph process (also known as the finite-velocity random walk) describes the motion of a particle moving along a straight line at a constant speed, reversing its direction at random times governed by a Poisson process. This probabilistic model overcomes a key physical limitation of standard Brownian motion, which implies infinite propagation speed and non-differentiable trajectories. \\

Let us denote by $Y = \{Y_t\}_{t\geq 0}$ the telegraph process. The particle starts at position $Y_0$ with initial velocity $\mathcal{V}_0$, then we have a path representation for $t>0$,
\begin{align*}
Y_t = Y_0 + \int_0^t \mathcal{V}_s ds, \quad \mathcal{V}_s := \mathcal{V}_0 (-1)^{N_s}
\end{align*}
with
\begin{align*}
\mathbf{P}(-ct \leq Y_t -Y_0 \leq ct) = 1 \quad \textrm{for $t>0$ and} \quad \mathbf{P}(|\mathcal{V}_0|=c)=1.
\end{align*}
The speed is strictly bounded and the particle stays confined within a physical wavefront. The velocity is driven by a homogeneous Poisson process $N$ with intensity $\eta/\sigma$, alternating between $+c$ and $-c$ at every Poisson event. The detailed formula for the position $Y$ can be given as the sum of displacements over the random (independent and exponentially distributed) time intervals, that is
\begin{align}
Y_t = Y_0 + \mathcal{V}_0 \sum_{j=1}^{N_t} (-1)^{j-1} \big(T_j - T_{j-1} \big) + \mathcal{V}_0 (-1)^{N_t} \big( t - T_{N_t}\big)
\label{seriesTEL}
\end{align}
and
\begin{align*}
\mathbf{P}(T_j - T_{j-1} > t) = e^{-(\sigma/\eta) t}.
\end{align*}
The singular part of the probability measure represents the event of zero reversals, where the particle moves deterministically along the wavefronts. From the (linear) residual process, it follows that
\begin{align*}
\mathbf{P}(Y_t - Y_0 \in dx\, | \, N_t=0) = \frac{1}{2}\Big(\delta_{+ct}(dx) + \delta_{-ct}(dx) \Big)
\end{align*}
is written in terms of Dirac measures. The absolutely continuous part of the probability measure describes the particle's position on $(Y_0-ct, Y_0+ct)$. Thus, for the process $Y$ on $[x-ct, x+ct]$ and a function $f$ sufficiently regular,  
\begin{align*}
w(t,x) = \mathbf{E}[f(Y_t)\mathbf{1}_{(t< T_1)}\,|\, Y_0=x] + \mathbf{E}[f(Y_t)\mathbf{1}_{( t \geq T_1)}\,|\,Y_0=x] 
\end{align*}
solves the hyperbolic telegraph equation
\begin{align*}
(\sigma/\eta) \frac{\partial w}{\partial t}(t,x)  + \frac{\partial^2 w}{\partial t^2}(t,x) = c^2 \frac{\partial^2}{\partial x^2} w(t,x), \quad t>0,\, x \in \mathbb{R}
\end{align*}
with initial and boundary conditions
\begin{equation*}
\left\lbrace
\begin{array}{l}
\displaystyle w(0,x) = f(x), \\
\\
\displaystyle \frac{\partial w}{\partial t}(t, x) \bigg|_{t=0} = 0.
\end{array}
\right .
\end{equation*}
The equation bridges wave-like behaviour at short time scales (due to the second-order time derivative) and diffusive behaviour at long time scales, where the process scales to a classic Brownian motion under a suitable Kac-scaling limit as $(\eta/\sigma) \to \infty$, $c \to \infty$.\\

The fractional telegraph process (or fractional telegraph equation) is an extension of the classical model designed to capture complex physical phenomena where memory effects and long-range dependencies play a crucial role. It is widely used to model anomalous transport, heat conduction in porous media, and diffusion in disordered systems. This generalization is achieved by replacing the standard time derivatives in the classical equation with fractional derivatives (typically in the Caputo or Riemann-Liouville sense) of non-integer orders. Unlike the classical telegraph process (with exponential inter-arrival times and a "memoryless" property) the fractional version governs direction reversals using waiting times with heavy-tailed distributions (the Mittag-Leffler distribution). This introduces a long-term memory, meaning the particle's next move depends heavily on its past history. A direct generalization of the Kac's representation leads to
\begin{align*}
Y^\alpha_t
= & Y^\alpha_0 + \mathcal{V}_0 \int_0^t (-1)^{N^L_s} ds\\
= & Y^\alpha_0 + \mathcal{V}_0 \sum_{j=1}^{N^L_t} \int_{T^L_{j-1}}^{T^L_j} (-1)^{j-1}\, ds + \mathcal{V}_0 \int_{T^L_{N^L_t}}^t (-1)^{N^L_s} ds\\
= & Y^\alpha_0 + \mathcal{V}_0 \sum_{j=1}^{N^L_t} (-1)^{j-1} \big(T^L_j - T^L_{j-1} \big)  + \mathcal{V}_0 (-1)^{N^L_t} \big( t - T^L_{N^L_t} \big).
\end{align*}
The process $Y^\alpha=\{Y^\alpha_t\}_{t \geq 0}$ is still confined with bounded speed. Recently, in \cite[Section 5.3]{ToaRic}, $Y^\alpha$ is considered as a special case of an isotropic transport process.

Concerning the fractional telegraph equation of order $2\alpha \in (0,2]$, 
\begin{align*}
(\sigma/\eta) D^\alpha_t v(t,x) + D^{2\alpha}_t v(t,x) = c^2 \frac{\partial^2 v}{\partial x^2} (t,x), \quad t>0,\; x \in \mathbb{R}
\end{align*}
with initial and boundary conditions
\begin{equation*}
\left\lbrace
\begin{array}{ll}
\displaystyle v(0,x) = f(x), & \textrm{for $\alpha \in (0,1]$},\\
\\
\displaystyle \frac{\partial v}{\partial t}(t, x) \bigg|_{t=0} = 0, &  \textrm{to be considered only for $\alpha \in (1/2, 1]$},
\end{array}
\right .
\end{equation*}
the analytical representation of the solution for any $\alpha$ has been extensively investigated, beginning with the pioneering work of \cite{CASCAVAL2002145} and \cite{OrsBeg}, and later extended to the case of space-fractional operators in \cite{OrsZhao} and to higher-dimensional generalizations in \cite{OrsDeG}. In \cite{CASCAVAL2002145} and \cite{OrsBeg} the probabilistic representation for $0 < 2\alpha < 2$ is realized through a time-change approach from the standard process, that is, from the path perspective, it corresponds to consider the process
\begin{align*}
Y_{L_t} = Y_0 + \mathcal{V}_0 \sum_{j=1}^{N^L_t} (-1)^{j-1} \big(T_j - T_{j-1} \big) + \mathcal{V}_0 (-1)^{N^L_t} \big( L_t - T_{N^L_t}\big).
\end{align*}
The process $Y_{L} = \{Y_{L_t}\}_{t\geq 0}$ is considered as a random flight in the one-dimensional case in \cite[see formula (4.9)]{AngDeGGarIaf}. The continuous time $L_{t}$ introduces a typical "staircase behaviour" to the trajectories. The corresponding residual process maintains the memory through the term $(L_t - T_{N^L_t})$ and, due to the smooth distribution of $L_{t}$, does not introduce a singular part in the density measure. On the other hand, the density of the process $Y^\alpha$ maintains the singular part of the density and therefore the original finite-velocity wave nature of the telegraph process. 

To the best of our knowledge, specific literature focusing explicitly on this pathwise formulation appears to be scarce. Given the vastness of the field, the authors acknowledge that relevant contributions might exist elsewhere.\\

For $\xi \in \mathbb{R}$ and a sufficiently regular function $f$, the solution to
\begin{align*}
(\sigma/\eta) D^\alpha_t \hat{v}(t, \xi) + D^{2\alpha}_t \hat{v}(t, \xi) = -\xi^2 \hat{v}(t, \xi), \quad t >0
\end{align*}
with
\begin{equation*}
\hat{v}(0, \xi) = \hat{f}(\xi) \quad \textrm{for $\alpha \in (0, 1/2]$ \quad and} \quad \left\lbrace 
\begin{array}{l}
\hat{v}(0, \xi) = \hat{f}(\xi),\\
\partial_t \hat{v}(0, \xi) =0
\end{array}
\right. \quad \textrm{ for $\alpha \in (1/2, 1]$}
\end{equation*}
can be written in terms of the kernel $\mathcal{K}$, that is
\begin{align}
\hat{v}(t,\xi) = \hat{f}(\xi) \mathcal{E}_\alpha(t, \xi), \quad \mathcal{E}_\alpha(t, \xi) = \int_0^\infty \mathcal{K}(t, s) e^{-s \xi^2}ds
\label{mathcE}
\end{align}
where, for $\lambda>0$, $\xi \in \mathbb{R}$,
\begin{align*}
\int_0^\infty e^{-\lambda t} \left( \int_0^\infty \mathcal{K}(t, s) e^{-s \xi^2}ds\right) dt  = \frac{\lambda^{\alpha - 1} (\lambda^\alpha + \sigma/\eta)}{\lambda^\alpha (\sigma/\eta + \lambda^\alpha) + \xi^2}.
\end{align*}
We skip the calculation of $\mathcal{E}_\alpha$ and refer to \cite{CASCAVAL2002145} and \cite{OrsBeg} for details. We only recall that (formula (1.5) in \cite{OrsBeg})
\begin{align*}
\mathcal{E}_\alpha(t, \xi) = E^1(t, \xi) + E^2(t, \xi)
\end{align*}
where ($c=1$ in our discussion)
\begin{align*}
E^1(t, \xi) = \frac{1}{2} \left(1 + \frac{\sigma/\eta}{\sqrt{(\sigma/\eta)^2-c^2 \xi^2}}\right) E_\alpha((-\sigma/\eta + \sqrt{(\sigma/\eta)^2-c^2 \xi^2})t^\alpha) 
\end{align*}
\begin{align*}
E^2(t, \xi) = \frac{1}{2} \left(1 - \frac{\sigma/\eta}{\sqrt{(\sigma/\eta)^2-c^2 \xi^2}}\right) E_\alpha((-\sigma/\eta - \sqrt{(\sigma/\eta)^2-c^2 \xi^2})t^\alpha)
\end{align*}
and $E_\alpha(\cdot)$ is the Mittag-Leffler function. In the literature, different researchers have dealt with different functional settings. We focus our analysis on the anisotropic Sobolev space $H^{2\alpha, 1}$. In the recent work \cite{AkgPobPoz23}, the authors proved that the Carleman condition holds, which implies that the distribution of the telegraph process is uniquely determined by its moments. This result is strictly related to the strong continuity we discuss in the last part of the work. We also refer to the interesting contributions in \cite{FigCamChiCap08} and the more recent \cite{DeGGar25} and \cite{CinOrs23}. In \cite{GionaDCCK} the Markovian case is considered with special focus on the age representation. Furthermore, a generalized fractional partial differential equation involving a regularized operator related to the Prabhakar operator was considered in \cite{DovPol}. An alternative probabilistic representation of the solution to the telegraph-diffusion equation is established for the specific range of the fractional parameter $\alpha \in (0,1/2]$ (see \cite{DovOrsToa14}, \cite{DovToaOrs14}). In this case the path is obtained from the Brownian motion via time change with the inverse of the sum of stable subordinators. A connection with subordinators is also established in \cite{LLPMee} in terms of the d'Alembert formula for $\alpha \in (1,2]$.   \\

The telegraph process can be naturally conceptualized as a coupled continuous-time random walk. Since the literature on both the coupled and uncoupled (Montroll-Weiss) cases is quite vast, we refer the reader to the works \cite{Meer04,BioToa26} and \cite{ASTT26,SCALAS2006} for an exhaustive discussion and a closer look at financial applications. Our models involve resetting at each renewal epoch, we refer to \cite{ColPag} and the references therein for an overview of this setting.

\subsection{Main results}

In this section we first consider the sequence $\theta := \{\theta_j\}_{j \in \mathbb{N}_0}$ of random variables, and then introduce the Markov chain $\theta^* := \{\theta^*_j\}_{j \in \mathbb{N}}$ by sampling a continuous-time Markov chain at the renewal epochs.

\subsubsection{The processes $Z$ and $Z^\mathtt{tel}$}
We introduce the processes $Z=\{Z_t\}_{t \geq 0}$ and $Z^\mathtt{tel}=\{Z^\mathtt{tel}_t\}_{t\geq 0}$ defined by
\begin{align*}
Z_t = \sum_{j=1}^{N^L_t} \theta_j e_j \quad \textrm{and} \quad Z^\mathtt{tel}_t  = Z_t + \theta_{1 + N^L_t} \big(L_t - T_{N^L_t}\big)
\end{align*}
and establish their relationship with the local time of sticky Brownian motion. The residual process of the telegraph process can be either linear or non-decreasing, just as the boundary local time of a sticky Brownian motion can be measured on either the local time scale or the physical time scale.\\

The structural properties of $Z$ are introduced in Section~\ref{sec:Z}, while the continuous variant $Z^\mathtt{tel}$ is thoroughly detailed in Section~\ref{sec:Ztel}. We show that the processes $Z$ and $Z^\mathtt{tel}$ provide probabilistic solutions to the fractional telegraph equation. The coupled process $Z$ belongs to the class of coupled pure jump renewal processes (or coupled continuous-time random walks), where the jump amplitudes and the waiting times are intrinsically dependent through the sequence of underlying exponential variables $\{e_j\}_{j \in \mathbb{N}}$. Starting from the pure jump renewal process $Z_t$, we introduce its continuous-time evolution variant $Z^\mathtt{tel}$.  Within each renewal interval $[T^L_{j-1}, T^L_j)$, the particle moves along the path of $L$ with a random velocity $\theta_j = \pm 1$. For $\alpha \in (0,1]$ (see Theorem \ref{thm:Zpde} and Theorem \ref{thm:Ztelpde}), under standard boundary conditions, the expectations
\begin{align*}
u(t,x) = \mathbf{E}[f(x + Z_t)] \quad \textrm{and} \quad v(t,x) = \mathbf{E}[f(x + Z^\mathtt{tel}_t)]
\end{align*}
respectively solve
\begin{align*}
(\sigma/\eta) D^\alpha_t u(t,x)  + D^{2\alpha}_t u(t,x) = \frac{\partial^2}{\partial x^2} \bigg[ u(t,x) - f(x) \mathbf{P}(T^L_1 > t) \bigg], \quad t>0,\, x \in \mathbb{R}
\end{align*}
and 
\begin{align*}
(\sigma/\eta) D^\alpha_t v(t,x)  + D^{2\alpha}_t v(t,x) = \frac{\partial^2 v}{\partial x^2}(t,x), \quad t>0,\, x \in \mathbb{R}.
\end{align*}
The processes $Z$ and $Z^\mathtt{tel}$ capture different physical insights into the same underlying integrated telegraph-like motion, as visually contrasted in Figure~\ref{fig:uno} and Figure~\ref{fig:alpha} below. Moreover, 
\begin{align*}
c \, Z^\mathtt{tel} \textrm{ is equal in law to a fractional telegraph process $Y_{L}$}
\end{align*} 
in terms of their one-dimensional laws. Moreover, due to the local fluctuations of residual process, it holds that 
\begin{align*}
Z^\mathtt{tel} \textrm{ exhibits a stochastic velocity}.
\end{align*} 
 
\subsubsection{Sticky behaviour at different scales}

In Section \ref{sec:Sticky}  we establish a relationship between $Z^\mathtt{tel}$ and the difference of local times 
\begin{align*}
\gamma^+_{r^+_t}(a) - \gamma^+_{r^+_t}(b)
\end{align*}  
where the inverse process $r^+_t$ is such that
\begin{align*}
\gamma^+_{r^+_t} = t \quad \textrm{for} \quad \gamma^+_t = \gamma^+_t(a) + \gamma^+_t(b)
\end{align*}
and $\gamma^+_t$ is the total boundary local time of the reflected Brownian motion $X^+$ on $[a,b]$. While our results can be extended to cases where $X^+$ is defined on a smooth domain $\Omega$, we do not provide a fully generalized formulation here, but rather offer an exploratory discussion (Section \ref{sec:OmegaSmooth}). The core insight revolves around the local time of the process $\bar{X}$ which belongs to a class of sticky processes as defined in \cite{EJP}. Specifically, we write
\begin{align*}
\bar{X}_t = X^+_{\bar{V}^{-1}_t} \quad \textrm{where} \quad \bar{V}_t = t + H_{\gamma^+_t} 
\end{align*} 
and the process $\bar{X}$ is termed sticky, meaning that it spends a non-zero Lebesgue measure of time on the boundary. For $\alpha=1$ we write $X$ in place of $\bar{X}$ and $V_t = t + \gamma^+_t$, this process corresponds to the well-known Feller-Wentzell sticky Brownian motion. 

We mainly focus on the interval $[a,b]$ (and provide a brief discussion for balls of $\mathbb{R}^d$). To ensure symmetry in the dynamics, we consider a midpoint (or membrane with equidistant boundary layers for balls of dimension $d>1$), although we recognize the potential interest of the asymmetric case. Formula \eqref{equivXZ} is an example of a totally asymmetric dynamics. We prove (Theorem~\ref{thm:difUNO}) that
\begin{align*}
\textrm{for every fixed $t\geq 0$} \quad \gamma^+_{r^+_t}(a) - \gamma^+_{r^+_t}(b) \stackrel{d}{=} Z^\mathtt{tel}_t  \quad \textrm{with} \quad \alpha =1
\end{align*}
where $\gamma^+_{r^+_t}(a) - \gamma^+_{r^+_t}(b)$ is evaluated on the local time scale and $Z^\mathtt{tel}_t$ is driven by a linear residual process as expected (see formula \eqref{repW}). Then, we consider
\begin{align*}
\tilde{V}_t := \bar{V}_t - t = H_{\gamma^+_t}
\end{align*}
which is a residence time for the sticky Brownian motion $\bar{X}$ and
we prove (Theorem \ref{thm:difALPHA}) that
\begin{align*}
\textrm{for every fixed $t\geq 0$} \quad  \gamma^+_{\tilde{V}^{-1}_t}(a) - \gamma^+_{\tilde{V}^{-1}_t}(b) \stackrel{d}{=} Z^\mathtt{tel}_t  \quad \textrm{for} \quad \alpha \in (0,1].
\end{align*}
In this framework, the residual process of $Z^\mathtt{tel}$ is formulated in terms of the non-decreasing inverse $L$, accurately capturing the Mittag-Leffler time required by the telegraph process to reach the subsequent jump level. Recall that the residual process is linear for $Y^\alpha$. \\

In the last part of the work we obtain pathwise representations and replace the r.v. $\theta$ with a continuous time Markov chain $\theta^*$. We prove (Theorem \ref{thmDiffLT}) that
\begin{align}
\gamma^+_{r^+_t}(a) - \gamma^+_{r^+_t}(b) = \sum_{j=1}^{N_t} \theta^*_j e_j + \theta^*_{N_t + 1} \left( t - T_{N_t} \right)
\label{MarkChain}
\end{align}
and (Theorem \ref{thm:DifLTbar}) that
\begin{align}
\bar{\gamma}_{\bar{r}_t}(a) - \bar{\gamma}_{\bar{r}_t}(b) = \sum_{j=1}^{N^L_t} \theta^*_j e^L_j + \theta^*_{N^L_t+1} \left( t - T^L_{N^L_t} \right)
\label{TelMarkChain}
\end{align}
for a suitable choice of the Markov chain $\theta^*$ with state space $\{-1, +1\}$, see Section \ref{sec:Path}.\\ 

As a by-product of our results, for $\alpha \in (0,1]$, we obtain that for every fixed $t\geq 0$, 
\begin{align*}
\bar{\gamma}_{\bar{r}_t}(a) - \bar{\gamma}_{\bar{r}_t}(b) 
\end{align*}
is equal in distribution to
\begin{align*}
Y^\alpha_t - Y^\alpha_0 = \theta^*_0 \int_0^t (-1)^{N^L_s}\, ds
\end{align*}
where $\theta^*_0 = \mathcal{V}_0$ with $c=1$. Finally, for $\alpha \in (0,1]$, we obtain that 
\begin{align*}
\bar{\gamma}_{\bar{r}_t}(a) - \bar{\gamma}_{\bar{r}_t}(b) = \int_0^t \theta^*_{N^L_s + 1}\,  ds 
\end{align*}
holds pathwise, for all $t \geq 0$. From our perspective, this result is of particular interest, as it yields novel insights into the process $Y^{\alpha}$, which has hitherto been overlooked in the literature.\\

It is well known that the fractional telegraph process $Y_L$ converges to a time-changed Brownian motion under an appropriate hydrodynamic limit. From the process $Z^\mathtt{tel}$ and the representation \eqref{TelMarkChain} we know that the telegraph process (as the process associated with the fractional equation) is already linked to this limit process via the Markov chain $\theta^*$ (which is associated with $X^+$ and therefore with $\bar{X}$). This connection does not appear in $Y_L$. In Section \ref{sec:further} we provide more details. In addition, in Section \ref{sec:OmegaSmooth} we set up the framework for the analysis on balls of $\mathbb{R}^d$. \\

We underline that \eqref{TelMarkChain} gives the boundary local time (on the boundary local time scale) of the sticky process $\bar{X}$. In the case where $\theta^*=1$ a.s., we obtain the total boundary local time. For example, by recovering \eqref{MarkChain} and \eqref{TelMarkChain}, we can observe that the coupled renewal reward processes
\begin{align*}
\sum_{j=1}^{N_t} e_j, \quad \textrm{and} \quad \sum_{j=1}^{N^L_t} e^L_j
\end{align*} 
both exhibit jump sizes and waiting times of matching magnitudes. That is, the trajectory is given by jumps on a straight line constructed by joining segments according to the jumps of $\gamma^+_t$ and $\bar{\gamma}_t$, for which we indeed have
\begin{align*}
\gamma^+_{r^+_t} = t, \quad \textrm{and} \quad \bar{\gamma}_{\bar{r}_t}=t.
\end{align*} 
From these deterministic paths, it is not possible to discern the jumps of the underlying (elementary) subordinator.\\

Finally, we emphasize that the telegraph process $Y^\alpha$ possesses a linear residual process that determines its characteristic deterministic path, which is in good agreement with a process on the local time scale. On the other hand, the process $Z^\mathtt{tel}$ - as well as the process $Y_L$ - exhibits a non-linear path due to its non-linear residual process (Figure \ref{fig:Ztel} and Figure \ref{fig:ZtelPart}). We refer to $Y^\alpha$ as a fractional telegraph process, as it arises from the generalization of Kac's definition of the telegraph process.

\section{Observing jumps and holding times: the process $Z$}
\label{sec:Z}
Let us consider the problem
\begin{align}
(\sigma/\eta) D^\alpha_t u(t,x)  + D^{2\alpha}_t u(t,x) = \frac{\partial^2}{\partial x^2} \bigg[ u(t,x) - f(x) \mathbf{P}(T^L_1 > t) \bigg], \quad t>0,\, x \in \mathbb{R}, \quad \alpha \in (0,1]
\label{pdeALPHA}
\end{align}
with boundary conditions
\begin{equation}
\left\lbrace
\begin{array}{ll}
\displaystyle \frac{\partial u}{\partial t}(0, x)=g(x), \quad \alpha \in (1/2, 1], \quad g=0 \textrm{ a.e.}, \\
\\
\displaystyle u(0,x)=f(x), \quad \alpha \in (0,1], \quad f \in H^2(\mathbb{R}). 
\end{array}
\right.
\label{pdeBC}
\end{equation}
If $\alpha=1$, then \eqref{pdeALPHA} takes the form
\begin{align}
(\sigma/\eta) \frac{\partial u}{\partial t}(t,x)  + \frac{\partial^2 u}{\partial t^2}(t,x) = \frac{\partial^2}{\partial x^2} \bigg[ u(t,x) - f(x) \mathbf{P}(T_1 > t) \bigg], \quad t>0,\, x \in \mathbb{R}.
\label{pdeUNO}
\end{align}
Recall that 
\begin{align*}
\mathbf{P}(T^L_1> t) = E_\alpha(-(\sigma/\eta) t^\alpha)\quad \textrm{and}  \quad \mathbf{P}(T_1 > t) = e^{-(\sigma/\eta) t}.
\end{align*}
The solution $u(t,x)$, $t \in (0, \infty)$, $x \in \mathbb{R}$ to the problem \eqref{pdeALPHA}-\eqref{pdeBC} belongs to the weighted anisotropic space-time Sobolev space $H^{2\alpha, 1}_\mu((0, \infty) \times \mathbb{R})$, equipped with the norm
\begin{equation}
\|u\|_{H^{2\alpha, 1}_\mu((0, \infty) \times \mathbb{R})} = \left( \|u\|^2_{L^2((0, \infty) \times \mathbb{R}; \mu)} + \|D^{2\alpha}_t u\|^2_{L^2((0, \infty) \times \mathbb{R}; \mu)} + \left\|\frac{\partial u}{\partial x}\right\|^2_{L^2((0, \infty) \times \mathbb{R}; \mu)} \right)^{1/2} < \infty
\end{equation}
with 
\begin{align*}
\| \varphi \|^2_{L^2((0, \infty)\times \mathbb{R}; \mu)} = \int_0^\infty   t^\zeta e^{-\beta t} \| \varphi(t, \dot) \|^2_{L^2(\mathbb{R})} \, dt.
\end{align*}
We provide further analysis below.\\

We introduce the coupled process $Z=\{Z_t\}_{t \geq 0}$ with
\begin{align*}
Z_t = \sum_{j=0}^{N^L_t} \theta_j e_j, \quad t\geq 0. 
\end{align*}
Our coupling is now clearly realized by the fact that
\begin{align*}
T_n = \sum_{j=0}^n e_j \quad \textrm{implies} \quad T^L_n = H_{T_n} \quad \textrm{and} \quad Z_{T^L_n} = \sum_{j=0}^n \theta_j e_j, \quad n \in \mathbb{N}_0.
\end{align*}
The sequence $\{e^L_j\}_{j \in \mathbb{N}_0}$ is the sequence of waiting (or holding) times for $Z$ which is a pure jumping process. Recall that $e_0=0$ and therefore  $Z_0=0$. Thus, $H_{e_0}=0$ and $\{e^L_j\}_{j \in \mathbb{N}}$ is the sequence of i.i.d. random variables such that $e^L_j \stackrel{d}{=} H_{e_1}$ for $j \in \mathbb{N}$. The apparently redundant ratio $\eta/\sigma$ will have a clear reading in the discussion below about sticky boundary conditions.\\

\begin{theorem}
The solution $u$ to the problem \eqref{pdeALPHA}-\eqref{pdeBC} admits the representation 
\begin{align*}
u(t,x)=\mathbf{E}[f(x + Z_t)], \quad t>0,\, x \in \mathbb{R}.
\end{align*}
\label{thm:Zpde}
\end{theorem}
\begin{proof}
We evaluate
\begin{align*}
\mathbf{E} \left[ \int_0^\infty e^{-\lambda t} e^{-i\xi Z_t} dt \right] 
= & \mathbf{E} \left[ \int_0^\infty e^{-\lambda t} \sum_{n \geq 0} e^{-i \xi \sum_{j=0}^n \theta_j e_j} \mathbf{1}(T^L_n \leq t < T^L_{n+1}) dt \right]
\end{align*}
by taking into account the $\sigma$-field
\begin{align*}
\mathcal{F}^L_{n} = \sigma\bigg\{\{\theta_j\}_{j \leq n},\, \{e_j\}_{j \leq n},\, \{H_t\}_{0 \leq t \leq T_n}\bigg\}
\end{align*}
and we write
\begin{align*}
\mathbf{E} \left[ \int_0^\infty e^{-\lambda t} e^{-i\xi Z_t} dt \right] 
= & \mathbf{E} \left[ \sum_{n \geq 0} e^{-i \xi \sum_{j=0}^n \theta_j e_j} \int_{T^L_n}^{T^L_{n+1}} e^{-\lambda t} dt \right]\\
= & \sum_{n \geq 0} \mathbf{E} \left[ \mathbf{E} \left[  e^{-i \xi \sum_{j=0}^n \theta_j e_j} \int_{T^L_n}^{T^L_{n+1}} e^{-\lambda t} dt \Bigg| \mathcal{F}^L_n \right]  \right]\\
= & \sum_{n \geq 0} \mathbf{E} \left[ \mathbf{E} \left[ e^{-i \xi \sum_{j=0}^n \theta_j e_j - \lambda T^L_n}\Bigg| \mathcal{F}^L_n \right]  \frac{\left(1 - e^{-\lambda H_{e_{n+1}}} \right)}{\lambda}  \right]\\
= & \sum_{n \geq 0} \mathbf{E} \left[ \frac{\left(1 - e^{-\lambda H_{e_{n+1}}} \right)}{\lambda} \prod_{j=0}^n e^{-i \xi \theta_j e_j - \lambda H_{e_j}}   \right].
\end{align*}
Indeed, $T^L_n = H_{T_n}$ and 
\begin{align*}
T^L_{n+1} - T^L_{n} \indep T^L_n \quad \textrm{with} \quad T^L_{n+1} - T^L_{n} \stackrel{d}{=} H_{T_{n+1} - T_n} = H_{e_{n+1}}. 
\end{align*}
In particular \eqref{AllIndep} holds and we get
\begin{align*}
\mathbf{E} \left[ \int_0^\infty e^{-\lambda t} e^{-i\xi Z_t} dt \right] 
= & \sum_{n \geq 0} \mathbf{E} \left[ \frac{\left(1 - e^{-\lambda^\alpha e_{n+1}} \right)}{\lambda} \prod_{j=0}^n e^{-i \xi \theta_j e_j - \lambda^\alpha e_j}   \right]\\
= &  \frac{\lambda^{\alpha-1}}{\lambda^\alpha + (\sigma/\eta)} \sum_{n \geq 0} \prod_{j=0}^n \mathbf{E} \left[ e^{-i \xi \theta_j e_j - \lambda^\alpha e_j}   \right]\\
= & \left[ \textrm{recall that $e_0=0$} \right]\\
= &  \frac{\lambda^{\alpha-1}}{\lambda^\alpha + (\sigma/\eta)} \sum_{n \geq 0} \prod_{j=1}^n \mathbf{E} \left[ \frac{\sigma/\eta}{(\sigma/\eta + \lambda^\alpha - i \xi \theta_j)}   \right]\\
= & \frac{\lambda^{\alpha-1}}{\lambda^\alpha + (\sigma/\eta)} \sum_{n \geq 0} \left( \frac{\sigma/\eta (\sigma/\eta + \lambda^\alpha)}{(\sigma/\eta + \lambda^\alpha)^2 + \xi^2 }\right)^n \\
= & \frac{\lambda^{\alpha-1}}{\lambda^\alpha + (\sigma/\eta)} \frac{(\sigma/\eta + \lambda^\alpha)^2 + \xi^2}{\lambda^\alpha (\sigma/\eta + \lambda^\alpha) + \xi^2}.
\end{align*}
Thus, for 
\begin{align*}
\tilde{\hat{u}}(\lambda, \xi) := \int_0^\infty e^{-\lambda t} \int_{\mathbb{R}} e^{-i\xi x} \mathbf{E}[f(x + Z_t)] dx\, dt = \hat{f}(\xi) \mathbf{E} \left[ \int_0^\infty e^{-\lambda t} e^{-i\xi Z_t} dt \right], \quad \lambda>0,\; \xi \in \mathbb{R},
\end{align*}
where
\begin{align*}
\hat{f}(\xi) = \int_\mathbb{R} e^{-i \xi x} f(x) dx,
\end{align*}
we get
\begin{align*}
\tilde{\hat{u}}(\lambda, \xi) =  \frac{\lambda^{\alpha - 1}(\sigma/\eta + \lambda^\alpha) + \xi^2}{\lambda^\alpha (\sigma/\eta + \lambda^\alpha) + \xi^2} \hat{f}(\xi)  + \frac{\lambda^{\alpha-1}}{\lambda^\alpha + (\sigma/\eta)} \frac{\xi^2}{\lambda^\alpha (\sigma/\eta + \lambda^\alpha) + \xi^2} \hat{f}(\xi)
\end{align*}
and
\begin{align*}
(\eta/\sigma) \big( \lambda^\alpha \tilde{\hat{u}}(\lambda, \xi)  - \lambda^{\alpha-1}\hat{f}(\xi) \big) + \big(\lambda^{2\alpha} \tilde{\hat{u}}(\lambda, \xi)  - \lambda^{2\alpha -1} \hat{f}(\xi)\big) = - \xi^2 \bigg( \tilde{\hat{u}}(\lambda, \xi)  - \frac{\lambda^{\alpha-1}}{\lambda^\alpha + \sigma/\eta} \hat{f}(\xi) \bigg) 
\end{align*}
is the double transform associated to \eqref{pdeALPHA}-\eqref{pdeBC}.

\end{proof}

\section{Connecting the jump states: the process $Z^\mathtt{tel}$}
\label{sec:Ztel}
We consider the problem
\begin{align}
(\sigma/\eta) D^\alpha_t v(t,x)  + D^{2\alpha}_t v(t,x) = \frac{\partial^2 v}{\partial x^2} (t,x), \quad t>0,\, x \in \mathbb{R}, \quad \alpha \in (0,1]
\label{pdeTelALFHA}
\end{align}
with boundary conditions
\begin{equation}
\left\lbrace
\begin{array}{ll}
\displaystyle \frac{\partial v}{\partial t}(0, x)=g(x), \quad \alpha \in (1/2, 1], \quad g=0 \textrm{ a.e.}, \\
\\
\displaystyle v(0,x)=f(x), \quad \alpha \in (0,1], \quad f \in H^2(\mathbb{R}). 
\end{array}
\right.
\label{pdeTelBC}
\end{equation}
The solution $v(t, x)$, $t \in (0, \infty)$, $x \in \mathbb{R}$ to the homogeneous fractional telegraph problem \eqref{pdeTelALFHA}-\eqref{pdeTelBC} belongs to the weighted anisotropic space-time Sobolev space $H^{2\alpha, 1}_\mu((0, \infty) \times \mathbb{R})$, equipped with the norm
\begin{equation}
\|v\|_{H^{2\alpha, 1}_\mu((0, \infty) \times \mathbb{R})} = \left( \|v\|^2_{L^2((0, \infty) \times \mathbb{R}; \mu)} + \|D^{2\alpha}_t v\|^2_{L^2((0, \infty) \times \mathbb{R}; \mu)} + \left\|\frac{\partial v}{\partial x}\right\|^2_{L^2((0, \infty) \times \mathbb{R}; \mu)} \right)^{1/2} < \infty
\end{equation}
with 
\begin{align*}
\| \varphi \|^2_{L^2((0, \infty)\times \mathbb{R}; \mu)} = \int_0^\infty   t^\zeta e^{-\beta t} \| \varphi(t, \dot) \|^2_{L^2(\mathbb{R})} \, dt.
\end{align*}
Our aim is to find a stochastic solution for the fractional telegraph equation. Although the analytical solution is known to be unique, it admits different probabilistic representations.\\

Let us introduce, for $t\geq 0$, the process
\begin{align*}
Z_t^\mathtt{tel} =  \sum_{j=1}^{N^L_t} \theta_j e_j + \theta_{1 + N^L_t}\big( L_t - T_{N^L_t}\big) = Z_t + \theta_{1 + N^L_t} \big( L_t - T_{N^L_t}\big).
\end{align*}
Thus, the process $Z^\mathtt{tel} = \{Z^\mathtt{tel}_t\}_{t\geq 0}$ has continuous paths with $Z^\mathtt{tel}_0=0$ and, for $t>0$,   
\begin{align*}
Z^\mathtt{tel}_t = \left\lbrace
\begin{array}{ll}
\displaystyle \theta_1 L_t, & 0 \leq t < T^L_1 := H_{e_1},\\
\\
\displaystyle \theta_1 e_1 + \theta_2 \big(L_t - T_1 \big), & T^L_1 \leq t < T^L_2 := H_{e_1 + e_2}, \\
\\
\displaystyle  .... \\
\\
\displaystyle \sum_{j=1}^n \theta_j e_j + \theta_{n+1} \big(L_t - T_n \big), & T^L_n \leq t < T^L_{n+1} := H_{T_{n+1}}, \quad n \in \mathbb{N}_0 .
\end{array}
\right . 
\end{align*}
The residual process exhibits a non-decreasing path given by $L$ (see Figure \ref{fig:Ztel}). Observe that $Z^\mathtt{tel}$ slightly differs from $Y_{L}$ obtained from the representation \eqref{seriesTEL} via time-change with $L$. For example, 
\begin{align*}
|c \, Z^\mathtt{tel}_t| \leq c\,T^L_{N^L_t} +  c\,| L_t - T^L_{N^L_t} | = c \,L_t, \quad t\geq 0.
\end{align*}
and
\begin{align*}
|Y_{L_t} - Y_{L_0}| \leq c\,T^L_{N^L_t} +  c\,| L_t - T^L_{N^L_t} | = c \, L_t, \quad t\geq 0.
\end{align*}
While the trajectories of the processes may exhibit different geometric features - due to the action of $\theta_{n+1}$ in place of $(-1)^{n}$ in the residual process - we now show that such a difference does not emerge in a weak sense (i.e., at the distributional level). Indeed, we obtain that, for a function $f$ sufficiently regular,
\begin{align*}
\mathbf{E}[f(Y_{L_t}) \, | Y_{L_0}=x] = \mathbf{E}[ f(x + c Z^\mathtt{tel}_t)], \quad t >0,\; x \in \mathbb{R}.
\end{align*}

We first underline that $H$ exhibits strictly increasing paths. For the epochs $\{T_j\}_{j \in \mathbb{N}}$, given a path of $H$ we identify the sequence $\{H_{T_j}\}_{j \in \mathbb{N}}$ and the sequence $\{H_{T_j} - H_{T_{j-1}}\}_{j \in \mathbb{N}}$ of independent increments. Recall that $e^L_j \stackrel{d}{=} H_{T_j} - H_{T_{j-1}}$ for all $j \in \mathbb{N}$ with $T_{0}=0$ and $T^L_j = H_{T_j} \stackrel{d}{=} e^L_1 + \cdots + e^L_j$. From the path of $H$ we obtain information on the path of $L_t = \inf\{s \geq 0\,:\, H_s >t\}$. In particular, at the renewal epoch $T_j$ we have $H_{T_j}:=T^L_j$ for which $L_{T^L_j}= T_j$ and the independent increment $H_{T_j} - H_{T_{j-1}}$ is mapped onto the increment $L_{H_{T_j}} - L_{H_{T_{j-1}}} = T_j - T_{j-1}=: e_j$ for all $j \in \mathbb{N}$. The increments $L_{T^L_j} - L_{T^L_{j-1}}$ are independent due to the Poisson sampling on the path of $H$ which provides a Fractional Poisson sampling on the path of $L$. Recall that for the fractional Poisson process $N^L$ we have independent Mittag-Leffler waiting times, that is $\{e^L_j\}_{j \in \mathbb{N}}$. The path of $L_t$ can therefore be written as
\begin{align}
L_t = \sum_{j=1}^{N^L_t} e_j + L_t - \sum_{j=1}^{N^L_t} e_j = \sum_{j=1}^{N^L_t} e_j + \big(L_t - T_{N^L_t} \big)
\label{repL}
\end{align}
and we obtain the residual process $\big(L_t - T_{N^L_t}\big)$. The fractional telegraph process is then obtained as a random perturbation of \eqref{repL}.

\begin{theorem}
The solution $v$ to the problems \eqref{pdeTelALFHA}-\eqref{pdeTelBC} admits the representation
\begin{align*}
v(t,x) = \mathbf{E}[f(x + Z_t^\mathtt{tel})], \quad t >0,\; x \in \mathbb{R}.
\end{align*} 
\label{thm:Ztelpde}
\end{theorem}
\begin{proof}
We follow the proof of the Theorem \ref{thm:Zpde}. For $\lambda>0$, $\xi \in \mathbb{R}$ we write
\begin{align*}
\mathbf{E} \left[ \int_0^\infty e^{-\lambda t} e^{-i\xi Z^\mathtt{tel}_t} dt \right] 
= & \mathbf{E} \left[ \int_0^\infty e^{-\lambda t} e^{-i\xi Z_t - i \xi \theta_{N^L_t + 1} ( L_t - T_{N^L_t})} dt \right]\\
= & \sum_{n \geq 0} \mathbf{E} \left[  e^{-i\xi \sum_{j=1}^{n} \theta_j e_j + i \xi \theta_{n + 1} T_n} \int_{T^L_n}^{T^L_{n+1}} e^{ - \lambda t - i \xi \theta_{n + 1} L_t } dt \right]\\
= & \sum_{n \geq 0}  \mathbf{E} \left[  \mathbf{E} \left[  e^{-i\xi \sum_{j=1}^{n} \theta_j e_j + i \xi \theta_{n + 1} T_n}  \int_{T^L_n}^{T^L_{n+1}} e^{ - \lambda t - i \xi \theta_{n + 1} L_t } dt \bigg| \, \mathcal{F}^L_n \right] \right]
\end{align*}
with (see for example \cite{EJP} for the Lebesgue-Stieltjes integral with respect to $H$)
\begin{align*}
& \mathbf{E}\left[ \mathbf{E} \left[   e^{-i\xi \sum_{j=1}^{n} \theta_j e_j + i \xi \theta_{n + 1} T_n}\int_{T^L_n}^{T^L_{n+1}} e^{ - \lambda t - i \xi \theta_{n + 1} L_t } dt \bigg| \, \mathcal{F}^L_n \right] \right] \\
= & - \frac{1}{\lambda} \mathbf{E}\left[ \mathbf{E} \left[  e^{-i\xi \sum_{j=1}^{n} \theta_j e_j + i \xi \theta_{n + 1} T_n} \int_{T^L_n}^{T^L_{n+1}} e^{ - i \xi \theta_{n + 1} t } de^{-\lambda t} \bigg| \, \mathcal{F}^L_n \right] \right]\\
= & - \frac{1}{\lambda} \mathbf{E}\left[ \mathbf{E} \left[  e^{-i\xi \sum_{j=1}^{n} \theta_j e_j + i \xi \theta_{n + 1} T_n} \int_{T_n}^{T_{n+1}} e^{- i \xi \theta_{n + 1} t } de^{-\lambda H_t} \bigg| \, \mathcal{F}^L_n \right] \right]\\
= & - \frac{1}{\lambda} \mathbf{E}\left[ \mathbf{E} \left[  e^{-i\xi \sum_{j=1}^{n} \theta_j e_j + i \xi \theta_{n + 1} T_n} \int_{T_n}^{T_{n+1}} e^{- i \xi \theta_{n + 1} t } de^{-\lambda^\alpha t} \bigg| \, \mathcal{F}^L_n \right] \right]\\
= & \lambda^{\alpha-1} \mathbf{E}\left[ \mathbf{E} \left[ e^{-i\xi \sum_{j=1}^{n} \theta_j e_j + i \xi \theta_{n + 1} T_n} \int_{T_n}^{T_{n+1}} e^{ - \lambda^\alpha t- i \xi \theta_{n + 1} t } dt \bigg| \, \mathcal{F}^L_n \right] \right]\\
= & \lambda^{\alpha - 1} \mathbf{E}\left[ \mathbf{E}\left[ e^{-i\xi \sum_{j=1}^{n} \theta_j e_j + i \xi \theta_{n + 1} T_n}\frac{\big(e^{-(\lambda^\alpha + i\xi \theta_{n+1})T_{n}} - e^{-(\lambda^\alpha + i\xi \theta_{n+1}) T_{n+1}} \big)}{\lambda^\alpha + i \xi \theta_{n+1}}  \bigg| \, \mathcal{F}^L_n \right] \right]\\
= & \lambda^{\alpha - 1}  \mathbf{E}\left[ \mathbf{E}\left[ e^{-i\xi \sum_{j=1}^{n} \theta_j e_j -\lambda^\alpha T_n} \frac{\big(1 - e^{-(\lambda^\alpha + i\xi \theta_{n+1})e_{n+1}}\big)}{\lambda^\alpha + i \xi \theta_{n+1}}  \bigg| \, \mathcal{F}^L_n \right] \right]\\
= & \lambda^{\alpha - 1} \mathbf{E}\left[ \frac{\big(1 - e^{-(\lambda^\alpha + i\xi \theta_{n+1})e_{n+1}}\big)}{\lambda^\alpha + i \xi \theta_{n+1}} \right] \mathbf{E}\left[ e^{-i\xi \sum_{j=1}^{n} \theta_j e_j -\lambda^\alpha T_n}  \right].
\end{align*}
Thus, the double transform takes the form
\begin{align*}
\mathbf{E} \left[ \int_0^\infty e^{-\lambda t} e^{-i\xi Z^\mathtt{tel}_t} dt \right] 
= & \sum_{n \geq 0} \mathbf{E}\left[ e^{-i\xi \sum_{j=1}^{n} \theta_j e_j -\lambda^\alpha T_n} \right] \lambda^{\alpha -1} \mathbf{E}\left[ \frac{\big(1 - e^{-(\lambda^\alpha + i\xi \theta_{n+1})e_{n+1}}\big)}{\lambda^\alpha + i \xi \theta_{n+1}}\right]\\
= & \sum_{n \geq 0} \left( \prod_{j=1}^n \mathbf{E}\left[ e^{-i\xi \theta_j e_j -\lambda^\alpha e_j} \right]\right) \lambda^{\alpha -1} \mathbf{E}\left[ \frac{\big(1 - e^{-(\lambda^\alpha + i\xi \theta_{n+1})e_{n+1}}\big)}{\lambda^\alpha + i \xi \theta_{n+1}}\right]
\end{align*}
where the distribution of $\theta_j$, $j \in \mathbb{N}$ plays a crucial role. The last expectation gives 
\begin{align*}
\mathbf{E}\left[ \frac{\big(1 - e^{-(\lambda^\alpha + i\xi \theta_{n+1})e_{n+1}}\big)}{\lambda^\alpha + i \xi \theta_{n+1}}\right] = \frac{(\sigma/\eta + \lambda^\alpha)}{(\sigma/\eta + \lambda^\alpha)^2 + \xi^2}.
\end{align*}
Now recall that
\begin{align*}
\mathbf{E} \left[ \int_0^\infty e^{-\lambda t} e^{-i\xi Z_t} dt \right] = \frac{\lambda^{\alpha-1}}{\lambda^\alpha + (\sigma/\eta)} \mathcal{Z}(\lambda, \xi) 
\end{align*}
where
\begin{align*}
\mathcal{Z}(\lambda, \xi) :=\frac{(\sigma/\eta + \lambda^\alpha)^2 + \xi^2}{\lambda^\alpha (\sigma/\eta + \lambda^\alpha) + \xi^2}, \quad \lambda>0,\; \xi \in \mathbb{R}.
\end{align*}
Thus, we get
\begin{align}
\mathbf{E} \left[\int_0^\infty e^{-\lambda t} e^{-i\xi Z^\mathtt{tel}_t} dt \right]= \mathcal{Z}(\lambda, \xi) \frac{\lambda^{\alpha -1} (\sigma/\eta + \lambda^\alpha)}{(\sigma/\eta + \lambda^\alpha)^2 + \xi^2} = \frac{\lambda^{\alpha -1} (\sigma/\eta + \lambda^\alpha)}{\lambda^\alpha (\sigma/\eta + \lambda^\alpha) + \xi^2}, \quad \lambda >0,\, \xi \in \mathbb{R}.
\label{ZteldoubleTransform}
\end{align}
It follows that
\begin{align*}
\tilde{\hat{v}}(\lambda, \xi) := \int_0^\infty e^{-\lambda t} \int_\mathbb{R} e^{-i\xi x} \mathbf{E}[f(x + Z_t^\mathtt{tel})]dx dt = \hat{f}(\xi) \mathbf{E}\left[ \int_0^\infty e^{-\lambda t} e^{-i \xi Z_t^\mathtt{tel}} dt \right]
\end{align*}
and
\begin{align*}
(\sigma/\eta) \big( \lambda^\alpha \tilde{\hat{v}}(\lambda, \xi) - \lambda^{\alpha-1} \hat{f}(\xi) \big) + \big(\lambda^{2\alpha} \tilde{\hat{v}}(\lambda, \xi) - \lambda^{2\alpha - 1} \hat{f}(\xi) \big) = - \xi^2 \tilde{\hat{v}}(\lambda, \xi).
\end{align*}
From \eqref{LAPder} we conclude the proof.
\end{proof}

Now we define
	\begin{equation}
		H_t^{(n)}:=H_{T_n+t}-H_{T_n},\quad t\ge0,
		\label{shiftH}
	\end{equation}
	and let
	\begin{equation}
		L_t^{(n)}:=\inf\{s\ge0:H_s^{(n)}>t\},\quad t\ge0.
		\label{shiftL}
	\end{equation}
	Notice that $e^L_{n+1} = H_{e_{n+1}}^{(n)}$, $n \in \mathbb{N}$ and the previous setting still holds. 

\begin{corollary}
We have that
\begin{align*}
Z_t^{\mathtt{tel}} = \sum_{j=0}^n\theta_j e_j + \theta_{n+1}L_{t-T_n^L}^{(n)}, \quad T_n^L\le t<T_{n+1}^L.
		\label{defZtel}
\end{align*}
\end{corollary}

\begin{proof}
Since an $\alpha$-stable subordinator $H$ is almost surely strictly increasing (and $T_n$ is independent of $H$), for $T_n^L\le t<T_{n+1}^L$ one has almost surely
	\begin{equation}
		L_{t-T_n^L}^{(n)}=L_t-T_n
		\label{Ln}
	\end{equation}
 which defines the residual process associated with $Z^\mathtt{tel}$. Once \eqref{Ln} is established, we can also identify the process $Z$. This representation holds pathwise for $t \geq 0$. 
\end{proof}

Recall $\mathcal{E}_\alpha$ introduced in \eqref{mathcE}. Since $(-\sigma/\eta + \sqrt{(\sigma/\eta)^2-c^2 \xi^2}) \leq 0$, the function $t \to \mathcal{E}_\alpha(t, \xi)$ is bounded. We recall that, for $\alpha \in (0,1]$, $E_\alpha(-z)$ with $z\geq 0$ is a continuous monotone function with 
\begin{align*}
\lim_{z\to \infty} E_\alpha(-z)=0, \quad \lim_{z \to 0}E_\alpha(-z)=1.
\end{align*}
Thus, we obtain $\mathcal{E}_\alpha(t, \xi) \to 1$ as $t \to 0$ for $c^2\xi^2 \neq (\eta/\sigma)^2$, for $\alpha \in (0,1]$. A further investigation reveals, for $\xi \in \mathbb{R}$, that $\mathcal{E}_\alpha(t, \xi) \to 1$ as $t \to 0$ for all $\alpha \in (0,1]$. \\

We can now establish the well-posedness and continuity of the solution in $H^1(\mathbb{R})$. Recall that $f \in H^1(\mathbb{R})$ and $v(t,x)$ is the solution obtained via the inverse Fourier transform of $\hat{v}(t,\xi) = \hat{f}(\xi)\mathcal{E}_\alpha(t,\xi)$.

\begin{lemma}
The map $t \mapsto v(t, \cdot)$ is continuous in $H^1(\mathbb{R})$ up to the initial time $t=0$, 
\begin{align*}
\lim_{t \to 0^+} \|v(t, \cdot) - f\|_{H^1(\mathbb{R})} = 0.
\end{align*}
\label{lemma:contZERO}
\end{lemma}

\begin{proof}
By applying Plancherel's identity, 
\begin{equation}
\|v(t, \cdot) - f\|_{H^2(\mathbb{R})}^2 = \int_{\mathbb{R}} (1 + \xi^2) \left| \mathcal{E}_\alpha(t, \xi) - 1 \right|^2 | \hat{f}(\xi) |^2 d\xi.
\label{eq:integral}
\end{equation}
To pass the limit $t \to 0^+$ inside the integral, we apply the Lebesgue's dominated convergence theorem. 

As mentioned above (pointwise convergence),  for any fixed $\xi \in \mathbb{R}$, we have $\lim_{t \to 0^+} \mathcal{E}_\alpha(t, \xi) = 1$ and consequently,
\begin{align*}
\lim_{t \to 0^+} (1 + \xi^2) \left| \mathcal{E}_\alpha(t, \xi) - 1 \right|^2 | \hat{f}(\xi) |^2 = 0 \quad \forall \xi \in \mathbb{R}
\end{align*}
Since the Mittag-Leffler functions composing the kernel $\mathcal{E}_\alpha$ are analytic and uniformly bounded for bounded times $t \in [0, T]$, there exists a constant $C > 0$ independent of $t$ and $\xi$ such that $\left|\mathcal{E}_\alpha(t, \xi)\right| \le C$. This implies that
\begin{align*}
    \left| \mathcal{E}_\alpha(t, \xi) - 1 \right|^2 \le (C + 1)^2.
\end{align*}
Thus, the integrand in \eqref{eq:integral} can therefore be upper-bounded by
\begin{align*}
   (1 + \xi^2) \left| \mathcal{E}_\alpha(t, \xi) - 1 \right|^2 | \hat{f}(\xi) |^2 \le (C+1)^2 (1 + \xi^2) | \hat{f}(\xi) |^2
\end{align*}
with
\begin{align*}
\int_\mathbb{R} (1 + \xi^2) | \hat{f}(\xi) |^2 dx  = \| f\|^2_{H^1(\mathbb{R})} < \infty.
\end{align*}
Finally we obtain
\begin{align*}
\lim_{t \to 0^+} \|v(t, \cdot) - f\|_{H^2(\mathbb{R})}^2 = \int_{\mathbb{R}} (1 + \xi^2) \left[ \lim_{t \to 0^+} \left| \mathcal{E}_\alpha(t, \xi) - 1 \right|^2 \right] | \hat{f}(\xi) |^2 d\xi = 0
\end{align*}
This concludes the proof of strong continuity in $H^1(\mathbb{R})$.
\end{proof}

We say that $v \in H^{2\alpha, 1}_\mu((0, \infty) \times \mathbb{R})$ meaning that
\begin{align*}
\|v\|_{H^{2\alpha, 1}_\mu}^2 = \int_0^\infty t^\zeta e^{-\beta t} \left( \|v(t, \cdot)\|^2_{L^2(\mathbb{R})} + \|D^{2\alpha}_t v(t, \cdot)\|^2_{L^2(\mathbb{R})} + \left\|\frac{\partial v}{\partial x}(t, \cdot)\right\|_2^2 \right) dt < \infty.
\end{align*}

\begin{lemma}
The solution $v$ belongs to the weighted anisotropic Sobolev space $H^{2\alpha, 1}_\mu((0, \infty) \times \mathbb{R})$ if and only if 
\begin{align*}
\alpha \in (0,1], \quad \beta > 0 \quad \text{and} \quad \zeta > - 1.
\end{align*}
\end{lemma}

\begin{proof}
To prove that $v \in H^{2\alpha, 1}_\mu((0, \infty) \times \mathbb{R})$, we must verify that
\begin{align*}
\|v\|^2_{H^{2\alpha, 1}_\mu} < \infty.
\end{align*}
By virtue of Plancherel's theorem, the norm of any function $\varphi(t,x)$ in the weighted space $L^2((0, \infty)\times \mathbb{R}; \mu)$ can be equivalently evaluated as
\begin{align*}
\|\varphi\|^2_{H^{2\alpha, 1}_\mu} = \int_0^\infty t^\zeta e^{-\beta t} \|\hat{\varphi}(t, \dot)\|^2_{L^2(\mathbb{R})} \, dt
\end{align*}
where $\hat{\varphi}(t, \xi)$ is the Fourier transform of $\varphi(t,x$.\\

({\it Step 1}) Recall $|\mathcal{E}_\alpha(t, \xi)| \le C$ as observed above. Then, we have
\begin{align*}
\|\hat{v}(t, \cdot)\|^2_{L^2(\mathbb{R}_\xi)} \le C^2 \|\hat{f}\|^2_{L^2(\mathbb{R})} = C^2 \|f\|^2_{L^2(\mathbb{R})}
\end{align*}
and
\begin{align*}
\|v\|^2_{H^{2\alpha, 1}_\mu} \le C^2 \|f\|^2_{L^2(\mathbb{R})} \int_0^\infty t^\zeta e^{-\beta t} \, dt = C^2 \|f\|^2_{L^2(\mathbb{R})} \frac{\Gamma(\zeta + 1)}{\beta^{\zeta + 1}}
\end{align*}
Since $f \in H^1(\mathbb{R}) \subset L^2(\mathbb{R})$, $\beta > 0$, and $\zeta > -1$ (ensuring the integrability of the singularity at $t \to 0^+$ via the Gamma function $\Gamma$), the first term is strictly finite.\\

({\it Step 2}) Now observe that
\begin{align*}
|\xi \hat{v}(t, \xi)| = |\xi \mathcal{E}_\alpha(t, \xi) \hat{f}(\xi)| \le C |\xi \hat{f}(\xi)|
\end{align*}
By applying Plancherel's identity 
\begin{align*}
\left\|\frac{\partial v}{\partial x}(t, \cdot)\right\|^2_{L^2(\mathbb{R})} = \int_{\mathbb{R}} \xi^2 |\hat{v}(t, \xi)|^2 \, d\xi \le C^2 \int_{\mathbb{R}} \xi^2 |\hat{f}(\xi)|^2 \, d\xi = C^2 \left\|\frac{\partial f}{\partial x}\right\|^2_{L^2(\mathbb{R})}
\end{align*}
and by integrating over the time domain we get
\begin{align*}
\left\|\frac{\partial v}{\partial x}\right\|^2_{H^{2\alpha, 1}_\mu} \le C^2 \left\|\frac{\partial f}{\partial x}\right\|^2_{L^2(\mathbb{R})} \int_0^\infty t^\zeta e^{-\beta t} \, dt = C^2 \left\|\frac{\partial f}{\partial x}\right\|^2_{L^2(\mathbb{R})} \frac{\Gamma(\zeta + 1)}{\beta^{\zeta + 1}}
\end{align*}
This expression is finite since $f \in H^1(\mathbb{R})$.\\

({\it Step 3})
Observe that
\begin{align*}
D^{2\alpha}_t \hat{v}(t, \xi) = -\xi^2 \hat{v}(t, \xi) - (\sigma/\eta) D^\alpha_t \hat{v}(t, \xi)
\end{align*}
under the algebraic inequality $|a + b|^2 \le 2|a|^2 + 2|b|^2$ leads to
\begin{align*}
\|D^{2\alpha}_t \hat{v}(t, \cdot)\|^2_{L^2(\mathbb{R})} \le 2 \|\xi^2 \hat{v}(t, \cdot)\|^2_{L^2(\mathbb{R})} + 2(\sigma/\eta)^2 \|D^\alpha_t \hat{v}(t, \cdot)\|^2_{L^2(\mathbb{R})}
\end{align*}
where
\begin{align*}
\|\xi^2 \hat{v}(t, \cdot)\|^2_{L^2(\mathbb{R})} = \int_\mathbb{R} |\xi^2 \mathcal{E}_\alpha(t,\xi) \hat{f}(\xi) |^2 d\xi \leq C^2 \left\|\frac{\partial^2 f}{\partial x^2}\right\|^2_{L^2(\mathbb{R})}.
\end{align*}
Observe that $D^\alpha_t E_\alpha(- q t^\alpha)= - q E_\alpha(-q t^\alpha)$ implies 
\begin{align*}
D^\alpha_t \mathcal{E}_{\alpha}(t,\xi) = \frac{\xi^2}{2 \sqrt{(\sigma/\eta)^2 - \xi^2}} \left( E_\alpha(q_1 t^\alpha) - E_\alpha(q_2 t^\alpha) \right)
\end{align*}
where $q_1 = -\sigma/\eta + \sqrt{(\sigma/\eta)^2-\xi^2}$ and $q_2 = -\sigma/\eta - \sqrt{(\sigma/\eta)^2-\xi^2}$, that is
\begin{align*}
D^\alpha_t \mathcal{E}_{\alpha}(t,\xi) = \frac{\xi^2}{2} \frac{E_\alpha(q_1 t^\alpha) - E_\alpha(q_2 t^\alpha)}{q_1 -q_2}.
\end{align*}
Since
\begin{align*}
(\sigma/\eta) \frac{E_\alpha(q_1t^\alpha) - E_\alpha(q_2t^\alpha)}{q_1 - q_2} = \mathcal{E}_\alpha(t, \xi) - \frac{E_\alpha(q_1 t^\alpha) + E_\alpha(q_2 t^\alpha)}{2}
\end{align*}
we conclude that
\begin{align*}
(\sigma/\eta) D^\alpha_t \mathcal{E}_{\alpha}(t,\xi) \leq \frac{\xi^2}{2} \left[ \mathcal{E}_\alpha(t, \xi) - \frac{E_\alpha(q_1 t^\alpha) + E_\alpha(q_2 t^\alpha)}{2} \right]
\end{align*}
where the addends in the right-hand side are bounded. Thus, there exist $\tilde{C}>0$ such that 
\begin{align*}
(\sigma/\eta)  D^\alpha_t \mathcal{E}_{\alpha}(t,\xi) \leq \tilde{C}\, \xi^2
\end{align*}
and
\begin{align*}
(\sigma/\eta)^2 \|D^\alpha_t \hat{v}(t, \cdot)\|^2_{L^2(\mathbb{R})} \leq \tilde{C} \int_\mathbb{R} \xi^4 |\hat{f}(\xi)|^2 d\xi = \tilde{C} \left\|\frac{\partial^2 f}{\partial x^2}\right\|^2_{L^2(\mathbb{R})}.
\end{align*}
Finally, we obtain 
\begin{align*}
\|D^{2\alpha}_t \hat{v}(t, \cdot)\|^2_{L^2(\mathbb{R})} \le 2 \big(C^2 + \tilde{C} \big) \left\|\frac{\partial^2 f}{\partial x^2}\right\|^2_{L^2(\mathbb{R})} < \infty
\end{align*}
due to $f \in H^2(\mathbb{R})$.\\

({\it Conclusion})
Combining the results of Steps 1, 2, and 3, all terms comprising the $H^{2\alpha, 1}_\mu((0, \infty) \times \mathbb{R})$ norm are finite, completing the proof.
\end{proof}

In conclusion, for $\alpha \in (0,1]$,
\begin{equation*}
v\in C((0,\infty );H^{1}(\mathbb{R}))
\end{equation*}
and possesses the necessary time-regularity for the fractional derivatives to be well-defined. The vanishing initial velocity ensures the continuity of the fractional framework up to the boundary $t=0$. \\

Lemma \ref{lemma:contZERO} implies that the underlying process is mean-square continuous at the initial instant, 
\begin{align*}
\lim _{t\rightarrow 0^{+}}\mathbf{E}[|Z^\mathtt{tel}_{t}-Z^\mathtt{tel}_{0}|^{2}]=0,
\end{align*}
which ensures that the process does not exhibit instantaneous macroscopic jumps. For the underlying fractional telegraph process, this regular behaviour ensures that the mean squared displacement, given by $\mathbf{E}[(Z^\mathtt{tel}_t)^2]$, is well-defined at the origin. Concerning the process $Z$, where the dynamics are resolved into waiting times (durations) and spatial jumps (amplitudes), this initial continuity carries a precise microscopic meaning. It implies that the waiting time distribution for the first event is well-behaved near the origin, preventing an unphysical accumulation of infinite jumps immediately after $t=0$. Consequently, the particle spends a physically consistent, non-vanishing duration in its initial state before memory effects and macroscopic transport features take over.

\begin{remark}
It is worth noting that we do not need to separately address the time-continuity at $t = 0$ for the solution $u(t, x)$ of Section \ref{sec:Z}. Indeed, one can follow exactly the same lines as established here for $v(t, x)$.
\end{remark}

\section{Sticky Brownian motions}
\label{sec:Sticky}
\subsection{Overview of the boundary behaviour}
In \cite{EJP} a wide class of sticky processes is defined via time change. In particular, a process is termed sticky if it spends a non-zero amount of Lebesgue time on the boundary. The present paper originated from our attempt to understand the fundamental dynamics of sticky processes. Our aim is simply to provide a first insight into this deep connection.\\

{\it (The case $\alpha=1$)} We consider $X=\{X_t\}_{t\geq 0}$ on the bounded domain $\Omega$ of $\mathbb{R}^d$, $d\geq 1$ and the reflected Brownian motion $X^+=\{X^+_t\}_{t \geq 0}$ for which
\begin{align*}
\mathbf{E}_x[f(X_t)] = \mathbf{E}_x[f(X^+ \circ V^{-1}_t)]
\end{align*}
with $V_t = t + (\eta/\sigma) \gamma^+_t$ where $\gamma^+_t$ is the local time of $X^+$ for which
\begin{align*}
\int_0^t \mathbf{1}_{\Omega}(X^+_s) d\gamma^+_s = 0.
\end{align*}
We equip the compact Euclidean space $\bar{\Omega} = \Omega \cup \partial \Omega$ with a finite Borel measure
\begin{align}
m(dx) = \mathbf{1}_\Omega\, dx + (\eta/\sigma)\, \mathbf{1}_{\partial \Omega} \, m_\partial(dx)
\label{mMeasure}
\end{align}
defined as the sum of the $d$-dimensional Lebesgue measure supported on the interior and the $(d-1)$-dimensional Hausdorff measure supported on the boundary. This implies that 
\begin{align*}
\gamma_t = \int_0^t \mathbf{1}_{\partial \Omega}(X_s)ds
\end{align*}
is not trivially zero. The sticky process $X$ has generator $(G, D(G))$ with $G=\Delta$ and
\begin{align*}
D(G) = \{\varphi, \Delta \varphi \in C(\bar{\Omega}), \, \varphi \in H^1(\Omega)\,:\, \eta (\Delta \varphi) |_{\partial \Omega} = - \sigma \partial_{\bf n} \varphi \}
\end{align*}
where $\partial_{\bf n}\varphi$ is the outer normal derivative with respect to $m$ introduced in \eqref{mMeasure}. We recall that $u|_{\partial \Omega}$ is the trace function (continuous operator) from $H^1(\Omega)$ into $L^2(\partial \Omega, m_\partial)$, such that $u \mapsto Tu=u|_{\partial \Omega}$ for $u \in H^1(\Omega)\cap C(\bar{\Omega})$. The reader can consult \cite{EJP} and the references therein for details. In the $1$-dimensional case $m_\partial(dx)$ is the Dirac measure $\delta_0(dx)$ if $\Omega=(0, \infty)$ or $\delta_a(dx) + \delta_b(dx)$ if $\Omega=(a,b)$. 

To fix ideas, we consider the $\epsilon$-neighbourhood 
\begin{align*}
\Lambda_\epsilon := \{x \in \bar{\Omega}\,:\, \operatorname{dist}(x, \partial \Omega) < \epsilon\}, \quad \epsilon>0
\end{align*}
and the exit time
\begin{align*}
\tau^+_\epsilon := \inf\{ t\,:\, X^+_t \notin \Lambda_\epsilon  \}, \quad \epsilon>0.
\end{align*}
For $X^+_0 \in \partial \Omega$, it follows that      
\begin{align*}
V \circ \tau^+_\epsilon = \tau^+_\epsilon + (\eta/\sigma) \gamma^+ \circ \tau^+_\epsilon =: \tau^+_\epsilon + e
\end{align*}
where, up to time $\tau^+_\epsilon$, 
\begin{itemize}
\item[i)] $\tau^+_\epsilon$ is the occupation time of $X^+$ on $\Lambda_\epsilon$;
\item[ii)] $\tau^+_\epsilon + (\eta/\sigma) \gamma^+ \circ \tau^+_\epsilon$ is the occupation time of $X$ on $\Lambda_\epsilon$.
\end{itemize}
In particular (this plays a crucial role),
\begin{align*}
\textrm{$e:=(\eta/\sigma) \gamma^+ \circ \tau^+_\epsilon$ is the extra time accumulated on $\partial \Omega$ by $X$ on $\Lambda_\epsilon$ up to time $\tau^+_\epsilon$.}
\end{align*}
We can introduce (\cite[Definition 6.3]{EJP}) the sequence $\{e_j\}_j = \{e_j,\, j \in \mathbb{N}\}$ of random variables as the sequence of sticky holding times for $X$. For a given $j \in \mathbb{N}$, $e_j$ denotes the time the process $X$ (started from $\partial \Omega$) spends on the boundary $\partial \Omega$ during the $j$-th visit to the $\epsilon$-neighbourhood $\Lambda_\epsilon$. Moreover, as a consequence of global regularity (e.g., when $\Omega$ is a ball) the sequence of sticky holding times $\{e_j\}_j$ consists of i.i.d. exponential random variables (see \cite[Remark 6.4 and Remark 6.5]{EJP}). Indeed, maintaining the assumption $e_0=0$, $e_j$ equals in law $(\eta/\sigma) \gamma^+_{\tau^+_\epsilon}$ for the visit $j \in \mathbb{N}$ to $\Lambda_\epsilon$ and
\begin{align*}
\textrm{$(\eta/\sigma) \gamma^+_\epsilon$ is an exponential r.v. with parameter depending on $\eta$, $\sigma$ and $\epsilon>0$.}
\end{align*}
The asymptotic regime as $\epsilon\to 0$ holds a pivotal role in the context of continuous-time random walks. In the local time scale - obtained by discarding the excursions into $\Omega$ (i.e. neglecting $\tau^+_\epsilon$) - the process
\begin{align}
\sum_{j=0}^{N_t} e_j, \quad t\geq 0
\label{equivXZ}
\end{align}
can be considered (see \cite[formula (6.10)]{EJP}) for the jumps of the local time $(\eta/\sigma) \gamma^+_t$ and therefore, to define the time accumulated by $X$ on the boundary. Moreover, if we consider in addition the residual process
\begin{align}
(\eta/\sigma) \big( \gamma^+_t - \gamma^+_{T_{N_t}} \big), \quad t\geq 0
\label{resiXZ}
\end{align}
then we can write $(\eta/\sigma)\gamma^+_t$ for $t\geq 0$ as the sum of \eqref{equivXZ} and \eqref{resiXZ} as discussed in \cite[formula (6.10)]{EJP}. Observe that, for all $t\geq 0$,
\begin{align*}
(\eta/\sigma) \big( \gamma^+_t - \gamma^+_{T_{N_t}} \big) = (\eta/\sigma) \gamma^+_{t -T_{N_t}} \quad \textrm{under $\mathbf{P}_x$ with $x=X^+_{T_{N_t}}$.}
\end{align*}
Under the assumption
\begin{align*}
\mathbf{P}(\theta_j = 1) = 1 \quad \textrm{for all $j \in \mathbb{N}_0$}
\end{align*}
for the process $Z$ with $\alpha=1$, we get exactly \eqref{equivXZ}. Thus, for  
the (Feller-Wentzell) sticky process $X$, the process $Z^\mathtt{tel}$ with $\alpha=1$ represents the local time (on the local time scale) accumulated on the boundary, which can be viewed as a piecewise linear trajectory composed of random segments with matching time and space displacements (i.e., the same waits and jumps). Even when successive jumps occur along the same line the underlying discrete structure of jumps and waiting times remains active. This hidden renewal mechanism guarantees that the accumulation of local time on the boundary is correctly recorded at the microscopic level. From the authors' viewpoint, this intimate relation between sticky local times and finite velocity models (for $\epsilon>0$) sheds new light on the semi-Markov models and the boundary trace processes (driven by Dirichlet-to-Neumann operators) in case of non-local dynamic boundary conditions. We study below this connection in the fractional order case  $\alpha \in (0,1]$.\\

{\it (The case $\alpha \in (0,1]$)} The class of sticky Brownian motion introduced in \cite{EJP} consists of processes associated to the non-local dynamic boundary value problems
\begin{align}
\alpha \in (0,1], \qquad \left\lbrace
\begin{array}{l}
\displaystyle \frac{\partial \varphi}{\partial t}(t,x) = \Delta \varphi(t,x), \quad t>0, \; x \in \Omega, \\
\\
\displaystyle \eta D^\alpha_t T \varphi(t,x) = - \sigma \partial_{\bf n} \varphi(t,x) - c \, \varphi(t,x), \quad t>0, \; x \in \partial \Omega,\\
\\
\displaystyle \varphi(0,x) = f(x), \quad x \in \bar{\Omega}.
\end{array}
\right .
\label{probXbar}
\end{align}
We relate the problem \eqref{probXbar} with $\bar{X}=\{\bar{X}_t\}_{t \geq 0}$ and the solution $\varphi(t,x) := \mathbf{E}_x[f(\bar{X}_t)]$ to \eqref{probXbar} can be written as 
\begin{align*}
\varphi(t,x) = \mathbf{E}_x[f(X^{el}_{\bar{V}^{-1}_t)}] = \mathbf{E}_x [f(X^+_{\bar{V}^{-1}_t}) \, M^+ \circ \bar{V}^{-1}_t], \quad \bar{X}_0=x
\end{align*}
where $\bar{V}^{-1}_t := \inf\{ s \, :\, \bar{V}_s > t\}$ is the inverse to the process
\begin{align}
\label{Vbar}
\bar{V}_t := t +  H_{(\eta/\sigma) \gamma^+_t}
\end{align}
and $M^+_t = e^{-(c/\sigma) \gamma^+_t}$ is the multiplicative functional associated to the elastic condition ($c=0$ in the present paper) and $\{X^{el}_t\}_{t\geq 0}$ is the elastic Brownian motion obtained via elastic killing from $X^+$. As we can see 
\begin{align*}
\bar{V}_{\tau^+_\epsilon} = \tau^+_\epsilon + H_{(\eta/\sigma)\gamma^+_{\tau^+_\epsilon}}
\end{align*}
where 
\begin{align*}
e:=H_{(\eta/\sigma)\gamma^+_{\tau^+_\epsilon}}
\end{align*}
is the extra time accumulated on the boundary by $\bar{X}$ during a useful (the process hits $\partial \Omega$) visit to $\Lambda_\epsilon$. This suggests that, for $e^L_j := T^L_j - T^L_{j-1} = H_{T_j} - H_{T_{j-1}}$, 
\begin{align}
\sum_{j=1}^{N^L_t} e^L_j, \quad t\geq 0
\label{equivHXZ}
\end{align}
(recall that $e_0=0$) is the associated pure jump process describing the local time accumulated on the boundary by $\bar{X}$ up to time $t$. This corresponds to the case of a totally asymmetric dynamics (see formula~\eqref{equivXZ} for $\alpha=1$). An interesting reading is given by the following formulas
\begin{align}
\sum_{j=1}^{N_t}  e_j +  \left( t - T_{N_t} \right) =  t
\label{tXpiu}
\end{align}
and
\begin{align}
\sum_{j=1}^{N^L_t} e^L_j + (t - T^L_{N^L_t}) =t
\label{tXbar}
\end{align}
where we used the fact that
\begin{align*}
\sum_{j=1}^{N^L_t} \big(H_{T_j} - H_{T_{j-1}}\big) = H_{T_1} + H_{T_2} - H_{T_1} + \cdots = H_{T_{N^L_t}} =: T^L_{N^L_t}.
\end{align*}
For example, we can say that $\bar{X}_{t}$ has spent a time $\sum_{j=1}^{N^L_t} e^L_j$ on the boundaries $\{a,b\}$ up to time $t$, and is bouncing at one of the two boundaries for a time $(t - T^L_{N^L_t})$, undergoing slow-downs or trappings. Formula \eqref{tXpiu} has a special meaning for $X_t$ based on the identity $\gamma^+_{r^+_t} = t$ as we discuss below. Analogously, formula \eqref{tXbar} can be associated with $\bar{\gamma}_{\bar{r}_t} = t$.\\

We are not interested in the asymptotic analysis ($\epsilon\to 0$); therefore, we do not refer to $\Lambda_\epsilon$ below, but rather consider a fixed neighbourhood. For simplicity, this neighbourhood is assumed to be symmetric, as introducing asymmetries would result in asymmetric dynamics.

\subsection{$\Omega$ is the interval $[a,b]$: a detailed analysis}

Consider the reflected Brownian motion $X^+$ restricted to the compact interval $[a, b]$ with reflecting boundaries at both endpoints. We focus on the specific geometric configuration where the interval is split exactly in half by its midpoint $(a+b)/2$. Let $\gamma^+_t(a)$ and $\gamma^+_t(b)$ denote the continuous local times accumulated at the boundaries $a$ and $b$ up to time $t$, respectively. Let $\gamma^+_t = \gamma^+_t(a) + \gamma^+_t(b)$ represent the total boundary local time of the process. Finally, we introduce
\begin{align*}
r^+_t := \inf\{s \geq 0\,:\, \gamma^+_s > t\}.
\end{align*}
Since $\gamma^+$ is continuous and non-decreasing, then $\gamma^+_{r^+_t}=t$. Observe that, for $t> 0$ we have
\begin{align*}
\mathbf{P}_0(X^+_t \leq (a+b)/2) = \mathbf{P}_0(X^+_t \geq (a+b)/2) = 1/2.
\end{align*}
We now introduce the sticky Brownian motion. Recall that the stopped local time is distributed as an exponential random variable (see for example \cite[Theorem 7.7]{ChugWilliams}). Recall that
\begin{align*}
\tau^+_0 := \inf\{t\,:\, X^+_t =0\}.
\end{align*}
We have two equivalent scenarios.
\begin{itemize}
\item[i)] We can consider the sticky process
\begin{align*}
X_t = X^+_{V^{-1}_t} \quad \textrm{on} \quad [a,b] = \Lambda_{1} := \{x \in \mathbb{R}\,:\, |x| \leq 1 \} \quad \textrm{with} \quad V_t = t + (\eta/\sigma) \gamma^+_t 
\end{align*}
where $(\eta/\sigma) \gamma^+_{\tau^+_0}$ is an exponential random variable with mean 
\begin{align*}
(\eta/\sigma)\mathbf{E}_x[\gamma^+_{\tau^+_0}] = \eta/\sigma, \quad x \in \{a, b\}.
\end{align*}
\item[ii)] We can consider the sticky process
\begin{align*}
X_t = X^+_{V^{-1}_t} \quad \textrm{on} \quad [a,b] = \Lambda_{\eta/\sigma} := \{x \in \mathbb{R}\,:\, |x| \leq \eta/\sigma \} \quad \textrm{with} \quad V_t = t + \gamma^+_t
\end{align*}
where $\gamma^+_{\tau^+_0}$ is an exponential random variable with mean 
\begin{align*}
\mathbf{E}_x[\gamma^+_{\tau^+_0}] = \eta/\sigma, \quad x \in \{a, b\}.
\end{align*}
\end{itemize}
In both cases, for $X^+_0 \in \{a, b\}$ or equivalently for $X_0 \in \{a, b\}$,  
\begin{align*}
\tau^+_0 \quad \textrm{and} \quad V_{\tau^+_0}
\end{align*}
provide the following information:
\begin{itemize}
\item $\tau^+_0$ is the occupation time of $X^+$ on $[a, 0)$ under $\mathbf{P}_a$ or the occupation time of $X^+$ on $(0, b]$ under $\mathbf{P}_b$;
\item $V_{\tau^+_0}$ is the occupation time of $X$ on $[a, 0)$ under $\mathbf{P}_a$ or the occupation time of $X$ on $(0, b]$ under $\mathbf{P}_b$.
\end{itemize}
 
We provide further details for the case ii) which is of interest further on. The process $X$ is a special case of the sticky process $\bar{X}$ on $[a,b]$. The $1$-dimensional process $\bar{X}$ on $[0, \infty)$ has been introduced in \cite{Dov22, NLBVP-ItheHline}. Here we consider the problem
\begin{align}
\left\lbrace
\begin{array}{l}
\displaystyle \frac{\partial \varphi}{\partial t}(t,x) = \varphi^{\prime \prime}(t,x), \quad t>0, \; x \in \Omega =(a,b), \\
\\
\displaystyle D^\alpha_t \varphi(t,a) = \varphi^\prime(t,a), \quad t>0, \\
\\
\displaystyle D^\alpha_t \varphi(t,b) = - \varphi^\prime(t,b), \quad t>0,\\
\\
\displaystyle \varphi(0,x) = f(x), \quad x \in \bar{\Omega} = [a,b].
\end{array}
\right .
\end{align}
with 
\begin{align*}
\varphi(t,x) = \mathbf{E}_x[f(\bar{X}_t)], \quad \bar{X}_t = X^+_{\bar{V}^{-1}_t}, \quad \bar{V}_t = t + H_{\gamma^+_t}, \quad \bar{X}_0=x
\end{align*}
for $t\geq 0$, $x \in [a,b]$ where $a=-\eta/\sigma$ and $b=\eta/\sigma$. This well agrees with item ii) above which is the case $\alpha=1$.

\subsubsection{The case $\alpha=1$}  We provide the following result for $X^+$ on $\Lambda_{\eta/\sigma}$. We consider the extra time $V_t -t$ for $X$ as defined in ii) of this section and observe that $t=V^{-1}_t + \gamma^+_{V^{-1}_t}$ and
\begin{align*}
\gamma_t := \int_0^t \mathbf{1}_{\{a, b\}}(X_s) ds = t - V^{-1}_t.
\end{align*}

\begin{theorem}
Fix $a = -(\eta/\sigma)$ and $b=+(\eta/\sigma)$. Assume $X^+_0 \sim Unif\{a,b\}$.  Then, 
\begin{align*}
\textrm{for each fixed $t\geq 0$} \quad \gamma^+_{r^+_t}(a) - \gamma^+_{r^+_t}(b) \stackrel{d}{=} Z^\mathtt{tel}_t  \quad \textrm{with} \quad \alpha =1.
\end{align*}
\label{thm:difUNO}
\end{theorem}

\begin{proof}
We aim to prove the identity in law between the unscaled time-changed local time difference $\gamma^+_{r^+_t}(a) - \gamma^+_{r^+_t}(b)$ and the coupled CTRW with a linear residue under the initial condition $\mathbf{P}(X^+_0 = a) = \mathbf{P}(X^+_0 = b) = \frac{1}{2}$. For the reflecting boundaries as $a = -\eta/\sigma$ and $b = \eta/\sigma$ the effective half-width of the interval is given by $\eta/\sigma$. The boundary local times satisfy the total mass identity $\gamma^+_{r^+_t}(a) + \gamma^+_{r^+_t}(b) = t$ for all $t \geq 0$ in the local time scale $t$, which yields directly
\begin{equation}
\gamma^+_{r^+_t}(a) - \gamma^+_{r^+_t}(b) = 2\gamma^+_{r^+_t}(a) - t.
\label{geom_red}
\end{equation}
Let $\zeta_0$ be the total local time accumulated on the boundary of origin before the process reaches the midpoint for the first time. In the local time scale $t$, the process hits the midpoint exactly at $t = \zeta_0$. From excursion theory, under $X^+_0 = a$, $\zeta_0$ follows an exponential distribution with parameter $m = \sigma/\eta$. We now evaluate
\begin{align*}
\varphi_x(\lambda, \xi):= \mathbf{E}_x\left[ \int_0^\infty e^{-\lambda t} e^{-i\xi \big( \gamma^+_{r^+_t}(a) - \gamma^+_{r^+_t}(b) \big)} dt\right], \quad x \in \{a,0,b\}, \quad \lambda>0, \; \xi \in \mathbb{R}.
\end{align*}
To compute $\varphi_a(\lambda, \xi)$, we decompose the expectation by conditioning on the hitting time $\zeta_0$. For $t < \zeta_0$, the process has not yet reached the origin, meaning 
\begin{align*}
\gamma^+_{r^+_t}(b) = 0 \quad \textrm{and} \quad \gamma^+_{r^+_t}(a) = t < \zeta_0.
\end{align*}
For $t \geq \zeta_0$, the strong Markov property ensures that the process restarts from the midpoint with an accumulated geometric offset $\gamma^+_{r^+_{\zeta_0}}(a) = \zeta_0$ on the left boundary, so that $\gamma^+_{r^+_t}(a) = \zeta_0 + \gamma^+_{r^+_{t-\zeta_0}}(a)$ under $X^+_{\zeta_0} = 0$. Thus, we can write
\begin{align}
\mathbf{E}_a\left[ e^{-\text{i}\xi \left(2\gamma^+_{r^+_t}(a) - t\right)} \right]  & = \mathbf{E}_a\left[ e^{-\text{i}\xi t} \mathbf{1}_{(t < \zeta_0)} \right] + \mathbf{E}_a\left[ e^{-\text{i}\xi \left(2\zeta_0 + 2\gamma^+_{r^+_{t-\zeta_0}}(a) - t\right)} \mathbf{1}_{(t \geq \zeta_0)} \right] \nonumber \\
&= e^{-\left(m + \text{i}\xi\right)t} + m \int_0^t e^{-(m + i \xi ) s} \mathbf{E}_0\left[ e^{-\text{i}\xi \left(2\gamma^+_{r^+_{t-s}}(a) - (t-s)\right)} \right] ds
\label{decom_integral}
\end{align}
where the second term is a convolution. Taking the Laplace transform of \eqref{decom_integral}, 
\begin{equation}
\varphi_a(\lambda, \xi) = \frac{1}{\lambda + \text{i}\xi + m} + \frac{m}{\lambda + \text{i}\xi + m} \varphi_0(\lambda, \xi).
\label{va_laplace}
\end{equation}
By symmetry, starting from $X^+_0 = b$ yields
\begin{equation}
\varphi_b(\lambda, \xi) = \frac{1}{\lambda - \text{i}\xi + m} + \frac{m}{\lambda - \text{i}\xi + m} \varphi_0(\lambda, \xi).
\label{vb_laplace}
\end{equation}
Since the process is immediately projected back to either boundary with equal probability upon hitting the midpoint, the renewal structure implies 
\begin{align*}
\varphi_0(\lambda, \xi) = \frac{1}{2} \varphi_a(\lambda, \xi) + \frac{1}{2} \varphi_b(\lambda, \xi).
\end{align*}
and we get
\begin{equation}
\varphi_0(\lambda, \xi) = \frac{\lambda+m}{\lambda (\lambda + m) + \xi^2}.
\label{v0_sol}
\end{equation}
From \eqref{va_laplace} and \eqref{vb_laplace}, averaging over the initial configurations $\mathbf{P}(X^+_0 = a) = \mathbf{P}(X^+_0 = b) = \frac{1}{2}$, we write
\begin{align*}
\frac{1}{2} \varphi_a(\lambda, \xi) + \frac{1}{2} \varphi_b(\lambda, \xi) = \frac{\lambda+m}{(\lambda+ m)^2 + \xi^2} \big( 1 + m \varphi_0(\lambda, \xi) \big) = \frac{\lambda + m}{\lambda^2 + m\lambda + \xi^2}.
\end{align*}
We obtain 
\begin{align*}
\mathbf{E} \left[ \mathbf{E}_{X^+_0} \left[ \int_0^\infty e^{-\lambda t} e^{-i\xi \big( \gamma^+_{r^+_t}(a) - \gamma^+_{r^+_t}(b) \big)} dt \right] \right] = \frac{\lambda + (\sigma/\eta)}{\lambda^2 + (\sigma/\eta) \lambda + \xi^2}
\end{align*}
which coincides with \eqref{ZteldoubleTransform} for $\alpha=1$. This completes the proof.
\end{proof}

\begin{remark}
We observe that
\begin{align*}
\mathbf{E} \left[ \mathbf{E}_{X^+_0} \left[ \int_0^\infty e^{-\lambda t} e^{-i\xi \big( \gamma^+_{r^+_t}(a) - \gamma^+_{r^+_t}(b) \big)} dt \right] \right] =  \mathbf{E}_{0} \left[ \int_0^\infty e^{-\lambda t} e^{-i\xi \big( \gamma^+_{r^+_t}(a) - \gamma^+_{r^+_t}(b) \big)} dt \right], \quad \lambda>0,\, \xi \in \mathbb{R}
\end{align*}
with $0=(a+b)/2$.
\end{remark}

\subsubsection{The case $\alpha \in (0,1]$}

For $X^+$ on $\Lambda_{\eta/\sigma}$, we now consider the sticky Brownian motion $\bar{X}=\{\bar{X}_t\}_{t \geq 0}$ with
\begin{align*}
\bar{X}_t = X^+_{\bar{V}^{-1}_t}, \quad \bar{V}_t = t + H_{\gamma^+_t}  
\end{align*}
and $\bar{V}_{\tau^+_\epsilon}$ plays the role of occupation time as described above for $\alpha=1$. Also in this case, 
\begin{align*}
\bar{V}_t - t, \quad t \geq 0
\end{align*}
determines the extra time on the boundary for $\bar{X}$. The extra time is the local time accumulated by $\bar{X}$ during the $j$-th visit of the process $X^+$ to the neighbourhood $\{x \in [a,b]\,:\, \operatorname{dist}(x, \{a,b\}) < \eta/\sigma\}$. Thus, $\bar{V}_t - t$ determines the residence times. Observe that $\bar{V}^{-1}_{\bar{V}_t}=t$ is always true and, for $\bar{X}_0= X^+_0  \in \{a,b\}$, 
\begin{align*}
\gamma^+_{\bar{V}^{-1}_{\bar{V}_{\tau^+_0}}} = \gamma^+_{\tau^+_0} \quad \textrm{whereas} \quad \bar{\gamma}_{\bar{V}^{-1}_{\tau^+_0}} = H_{\gamma^+_{\tau^+_0}}
\end{align*}
where the local time
\begin{align*}
\bar{\gamma}_t := \int_0^t \mathbf{1}_{\{a,b\}}(\bar{X}_s) ds = t - \bar{V}^{-1}_t
\end{align*}
coincides with the regulator $\tilde{\gamma}_t := \gamma^+_{\bar{V}^{-1}_t}$ only for $\alpha=1$. In the regulator-local-time scale we can not see $H$.\\

We consider the inverse to $\bar{V}_t - t = H_{\gamma^+_t}$ as a time change in the subsequent discussion. Let us define
\begin{align*}
\tilde{V}^{-1}_t = \inf\{s\,:\, \tilde{V}_s>t\} \quad \textrm{where} \quad \tilde{V}_t : = H_{\gamma^+_t}, \quad t \geq 0.
\end{align*}
We focus on the process
\begin{align*}
\gamma^+_{\tilde{V}^{-1}_t}(a) - \gamma^+_{\tilde{V}^{-1}_t} (b), \quad t\geq 0
\end{align*}
on the residence time scale.

\begin{theorem}
Fix $a = -(\eta/\sigma)$ and $b=+(\eta/\sigma)$. Assume $X^+_0 \sim Unif\{a,b\}$.  Then, 
\begin{align*}
\textrm{for each fixed $t\geq 0$} \quad  \gamma^+_{\tilde{V}^{-1}_t}(a) - \gamma^+_{\tilde{V}^{-1}_t} (b) \stackrel{d}{=} Z^\mathtt{tel}_t  \quad \textrm{with} \quad \alpha \in (0,1].
\end{align*}
\label{thm:difALPHA}
\end{theorem}
\begin{proof}
Let us write
\begin{align*}
D_t = D_t (a,b) := \gamma^+_t(a) - \gamma^+_t(b).
\end{align*}
Observe that
\begin{align*}
\mathbf{E}_0\left[ \int_0^\infty e^{-\lambda t} e^{-i\xi D_{\tilde{V}^{-1}_t}} dt \right] 
= & - \frac{1}{\lambda} \mathbf{E}_0\left[ \int_0^\infty  e^{-i\xi D_{\bar{V}^{-1}_t}} de^{-\lambda t} \right]\\
= & - \frac{1}{\lambda} \mathbf{E}_0\left[ \int_0^\infty  e^{-i\xi D_t} d \tilde{M}_t \right], \quad \tilde{M}_t = e^{-\lambda \tilde{V}_t}, \quad \lambda>0.
\end{align*}
Recall that $H \indep X^+$ and 
\begin{align*}
\mathbf{E}[e^{-\lambda \tilde{V}_t} | X^+_t] = e^{-\lambda^\alpha \gamma^+_t}.
\end{align*}
Thus, we get
\begin{align*}
\mathbf{E}_0\left[ \int_0^\infty e^{-\lambda t} e^{-i\xi D_{\tilde{V}^{-1}_t}} dt \right] 
= & \lambda^{\alpha - 1} \mathbf{E}_0\left[ \int_0^\infty e^{-\lambda^\alpha  \gamma^+_t} e^{-i\xi D_t} d \gamma^+_t \right]\\
= & \lambda^{\alpha - 1} \mathbf{E}_0\left[ \int_0^\infty e^{-\lambda^\alpha t} e^{-i\xi D_{r^+_t}} dt \right].
\end{align*}
From \eqref{v0_sol}, we obtain
\begin{align*}
\mathbf{E}_0\left[ \int_0^\infty e^{-\lambda t} e^{-i\xi D_{\tilde{V}^{-1}_t}} dt \right] 
= \frac{\lambda^{\alpha -1} (\lambda^\alpha + m)}{\lambda^\alpha (\lambda^\alpha + m) + \xi^2}.
\end{align*}
By following the proof of the previous theorem we get
\begin{align*}
\mathbf{E} \left[ \mathbf{E}_{X^+_0} \left[ \int_0^\infty e^{-\lambda t} e^{-i\xi D_{\tilde{V}^{-1}_t}} dt \right] \right] = \frac{\lambda^{\alpha -1} (\lambda^\alpha + (\sigma/\eta))}{\lambda^{2\alpha} + (\sigma/\eta) \lambda^\alpha + \xi^2}, \quad \lambda>0,\; \xi \in \mathbb{R}.
\end{align*}
This concludes the proof.
\end{proof}

\subsubsection{Pathwise decomposition and multi-scale analysis}
\label{sec:Path}
In this section, we develop a pathwise decomposition for the boundary local time. By performing a multi-scale analysis based on the inverse processes, we rigorously map the non-Markovian memory of the trapping boundary into a coupled continuous-time random walk.\\

Assume $X^+_0 \in \{a, b\}$ and consider the following setting:
\begin{itemize}
\item[i)] $T^*_j(\partial) := \inf \{t > T^*_j\,:\, X^+_t \in \{a,b\}\}$ for $j \in \mathbb{N}$ with $T^*_0(\partial)=0$;
\item[ii)] $T^*_j(0) := \inf\{t> T^*_{j-1}(\partial)\,:\, X^+_t =0\}$ for $j \in \mathbb{N}$. We simply write $T^*_j$; 
\item[iii)] $e^*_j := F\big(\gamma^+_{T^*_j}, \gamma^+_{T^*_{j-1}(\partial)}\big)$ for $j \in \mathbb{N}$;
\item[iv)] $N^*=\{N^*_t\}_{t \geq 0}$ with $N^*_t = \max\{j \in \mathbb{N}_0 \,:\, T^*_j < t\}$. In particular, 
\begin{align*}
\mathbf{P}(N^*_t \geq j) = \mathbf{P}(T^*_j < t)
\end{align*}
and
\begin{align*}
T^*_0 (\partial)=0, \quad \mathbf{P}(N^*_t=0) = \mathbf{P}(T^*_1 >t), \quad  T^*_1 = \tau^+_0 + F\big(\gamma^+_{T^*_1}, 0\big)
\end{align*}
with
\begin{align*}
\tau^+_0 = \inf\{t\,:\, X^+_t = 0 \} \quad \textrm{as defined above and} \quad e^*_1 = \gamma^+_{\tau^+_0};
\end{align*}
\item[v)] $\theta^*_{N^*_t + 1} = +1$ iff $X^+_{T^*_{N^*_t}(\partial)} = a$ and $\theta^*_{N^*_t + 1} = -1$ iff $X^+_{T^*_{N^*_t}(\partial)} =b$, that is 
\begin{align*}
\theta^*_{j} = \left\lbrace
\begin{array}{ll}
\displaystyle +1, & X^+_{T^*_{j-1}(\partial)} =a,\\
\\
\displaystyle -1, & X^+_{T^*_{j-1}(\partial)} =b,
\end{array}
\right . \quad j \in \mathbb{N}.
\end{align*}
In particular, $j \in \mathbb{N}$ denotes the $j$-th visit of the boundary $\{a, b\}$ and $\theta^*_j$ says in which boundary point the process accumulates local time in the renewal epoch $T^*_{j-1}$.
\end{itemize}

\begin{theorem}
Let the previous setting prevail. Then, 
\begin{align*}
\gamma^+_{r^+_t}(a) - \gamma^+_{r^+_t}(b) = \sum_{j=1}^{N_t} \theta^*_j e_j + \theta^*_{N_t + 1} \left( t - T_{N_t} \right).
\end{align*}
\label{thmDiffLT}
\end{theorem}

\begin{proof}
Assume $F(\gamma^+_{T^*_j},\gamma^+_{T^*_{j-1}(\partial)})= \gamma^+_{T^*_j} - \gamma^+_{T^*_{j-1}(\partial)}$ for $j \in \mathbb{N}$. Thus
\begin{align*}
e^*_j :=\gamma^+_{T^*_j} - \gamma^+_{T^*_{j-1}(\partial)} = \gamma^+_{T^*_j} - \gamma^+_{T^*_{j-1}}.
\end{align*}
Under i)-v) the difference between the two local times can be therefore rigorously expressed as
\begin{equation}
\gamma^+_t(a) - \gamma^+_t(b) = \sum_{j=1}^{N^*_t} \theta^*_j e^*_j + \theta^*_{N^*_t+1} \left( \gamma^+_t - \gamma^+_{T^*_{N^*_t}} \right) =: W_t
\label{repW}
\end{equation}
and the left- and right-hand side of \eqref{repW} live in the same probability space. Thus,  
\begin{align*}
\textrm{\eqref{repW} holds pathwise, for every $t\geq 0$.}
\end{align*}

For the sticky Brownian motion 
\begin{align*}
X=X^+_{V^{-1}_t} \quad \textrm{with} \quad V_t = t + \gamma^+_t
\end{align*}
started from one of the boundary points $a$ or $b$, the difference $T^*_j- T^*_{j-1}$ represents the occupation time on the corresponding neighbourhood of the type $\{x \in [a,b]\,:\, \operatorname{dist}(x, \{a,b\}) < \eta/\sigma\}$ in the renewal epoch $T^*_{j-1}$. Thus, during the $j$-th visit of the reflected Brownian motion $X^+$ to the neighbourhood of the boundary point,  the extra time $e^*_j:= \gamma^+_{T^*_j} - \gamma^+_{T^*_{j-1}}$ is the local time accumulated by $X$ at the renewal epoch $T^*_{j-1}$. Recall that
\begin{align*}
\gamma^+_{T^*_j} - \gamma^+_{T^*_{j-1}} = \gamma^+_{T^*_j} - \gamma^+_{T^*_{j-1}(\partial)} \indep  \gamma^+_{T^*_{j-1}(\partial)}
\end{align*}
from the strong Markov property of $X^+$. Assume $X^+_{T^*_{j-1}(\partial)} =a$, then $\gamma^+_{T^*_j} - \gamma^+_{T^*_{j-1}(\partial)}$ is the local time accumulated at the boundary point $a$ up to the first hitting time of the level $0$. This means that 
\begin{align*}
\mathbf{P}_a(\gamma^+_{T^*_j} - \gamma^+_{T^*_{j-1}(\partial)} > t ) = e^{-(\sigma/\eta) t}, \quad j \in \mathbb{N}.
\end{align*}
Analogously, if $X^+_{T^*_{j-1}(\partial)}=b$, then
\begin{align*}
\mathbf{P}_b(\gamma^+_{T^*_j} - \gamma^+_{T^*_{j-1}(\partial)} > t ) = e^{-(\sigma/\eta) t}, \quad j \in \mathbb{N}.
\end{align*}
In particular, $e^*_j \stackrel{d}{=} \gamma^+_{\tau^+_0}$ for all $j \in \mathbb{N}$. Observe that for the process $X^+$ on $[a,b]$  
\begin{align*}
\int_0^{\tau^+_0} \mathbf{1}_{\{a, b\}}(X^+_s)ds =0 
\end{align*}
whereas for the time accumulated by $X$ we have
\begin{align*}
V_{\tau^+_0} - \tau^+_0 = \int_0^{\tau^+_0} \mathbf{1}_{\{a, b\}}(X^+_s) d\gamma^+_s.
\end{align*}
We refer to $W=\{W_t\}_{t\geq 0}$ as a process in the physical time scale and consider the inverse to $V_t - t = \gamma^+_t$ as a time change in the subsequent discussion. \\

On the local time scale we have
\begin{equation}
\gamma^+_{r^+_t}(a) - \gamma^+_{r^+_t}(b) = \sum_{j=1}^{N^*_{r^+_t}} \theta^*_j e^*_j + \theta^*_{N^*_{r^+_t} + 1} \left( t - \gamma^+_{T^*_{N^*_{r^+_t}}} \right)=: W_{r^+_t}
\end{equation}
written also as
\begin{equation}
\gamma^+_{r^+_t}(a) - \gamma^+_{r^+_t}(b) = \sum_{j=1}^{N_t} \theta^*_j e_j + \theta^*_{N_t + 1} \left( t - T_{N_t} \right).
\label{DidLTscale}
\end{equation}
Indeed, the identity 
\begin{align*}
\gamma^+_{T^*_{N^*_{r^+_t}}} = T_{N_t}
\end{align*}
holds in a strict mathematical sense as a direct consequence of the continuity of the boundary local time and the duality between the local time and its inverse. Moreover, the jumps of $r^+_t$ give the time the process $X^+$ spends on the interior $(a,b)$, this implies that $N^*_{r^+_t}$ coincides with $N_t$. From the distribution point of view,
\begin{align*}
\mathbf{P}(N^*_{r^+_t} \geq j) = \mathbf{P}(T^*_j < r^+_t) = \mathbf{P}(\gamma^+_{T^*_j} < t)
\end{align*}
with 
\begin{align*}
\gamma^+_{T^*_j} - \gamma^+_{T^*_{j-1}} = \gamma^+_{T^*_j} - \gamma^+_{T^*_{j-1}} (\partial) = e_j \stackrel{d}{=} \operatorname{Exp}(\sigma/\eta) \quad \textrm{under $\mathbf{P}_x$ with $x=X^+_{T^*_{j-1}}$}
\end{align*}
In particular, $e^*_j=e_j$ for all $j \in \mathbb{N}$ where, we recall that,  
\begin{align*}
\mathbf{P}(N_t=0) = \mathbf{P}(T_1 >t), \quad T_1 = e_1
\end{align*}
and $T_j - T_{j-1}=e_j$, $j \in \mathbb{N}$.

\end{proof}

\begin{remark}
We observe that, for $\theta^*_j=1$ we get \eqref{tXpiu} which is now associated to $\gamma^+_{r^+_t} =t$.
\end{remark}

Let us consider
\begin{align*}
\bar{\gamma}_t(a) - \bar{\gamma}_t(b) \quad \textrm{where} \quad \bar{\gamma}_t = \bar{\gamma}_t(a) + \bar{\gamma}_t(b)
\end{align*}
is the boundary local time of the sticky Brownian motion $\bar{X}$ on $[a,b]$ and introduce the inverse $\bar{r}_t$ of the local time $\bar{\gamma}_t$, thus $\bar{\gamma}_{\bar{r}_t} = t$. This leads to the following fact:
\begin{center}
\begin{minipage}{0.8\textwidth}
    \itshape 
    ``For all $\alpha\in (0,1]$, although the relation $\bar{\gamma}_{\bar{r}_t} = t$ yields a deterministic straight line, it masks the underlying microstructure of the process, concealing the fact that both the duration of the plateaus and the magnitude of the jumps are identically scaled by $H_{\gamma^+_{\tau^+_0}}$.''
\end{minipage}
\end{center}
In this regard, we show that formula \eqref{tXbar} can be associated to $\bar{\gamma}_{\bar{r}_t}$. From this perspective, a comparative study of the differences
\begin{align*}
\bar{\gamma}_{\bar{r}_t}(a) - \bar{\gamma}_{\bar{r}_t}(b) \quad \text{and} \quad \gamma^+_{r^+_t}(a) - \gamma^+_{r^+_t}(b)
\end{align*}
becomes particularly intriguing. 

%

\begin{theorem}
Let the previous setting prevail. Then, 
\begin{align}
\bar{\gamma}_{\bar{r}_t}(a) - \bar{\gamma}_{\bar{r}_t}(b) = \sum_{j=1}^{N^L_t} \theta^*_j e^L_j + \theta^*_{N^L_t+1} \left( t - T^L_{N^L_t} \right).
\label{thmDifLTbar} 
\end{align}
\label{thm:DifLTbar}
\end{theorem}

\begin{proof}
Assume $F(\gamma^+_{T^*_j}, \gamma^+_{T^*_{j-1}(\partial)}) = H_{\gamma^+_{T^*_j}} - H_{\gamma^+_{T^*_{j-1}(\partial)}}$ for $j \in \mathbb{N}$. Thus, 
\begin{align*}
e^*_j := H_{\gamma^+_{T^*_j}} - H_{\gamma^+_{T^*_{j-1}(\partial)}} = H_{\gamma^+_{T^*_j}} - H_{\gamma^+_{T^*_{j-1}}} = H_{T_j} - H_{T_{j-1}} = T^L_j - T^L_{j-1} =: e^L_j
\end{align*}
where $T_{j-1}=\gamma^+_{T^*_{j-1}}$ is the renewal epoch on the local time scale. Under i)-v) the difference between the two local times can be therefore rigorously expressed as
\begin{align*}
\bar{\gamma}_t(a) - \bar{\gamma}_t(b) = \sum_{j=1}^{N^*_t} \theta^*_j e^*_j + \theta^*_{N^*_t+1} \left( \bar{\gamma}_t - \bar{\gamma}_{T^*_{N^*_t}} \right)
\end{align*}
where $\mathbf{P}(N^*_t \geq j) = \mathbf{P}(T^*_j <t)$. For example,
\begin{align*}
T^*_0=0 \quad \textrm{and} \quad T^*_1 = \tau^+_0 +   H_{\gamma^+_{\tau^+_0}} \textrm{ under $\mathbf{P}_x$ with $x \in \{a,b\}$}.    
\end{align*}
This means that
\begin{align*}
\bar{\gamma}_{T^*_1}(a) - \bar{\gamma}_{T^*_1}(b) = \theta^*_1\, \bar{\gamma}_{T^*_1} = \theta^*_1\, \bar{\gamma}_{T^*_1 (\partial)} =  \theta^*_1\, H_{\gamma^+_{\tau^+_0}}
\end{align*}
is the local time accumulated at the first visited boundary point. On the local time scale, 
\begin{align*}
\bar{\gamma}_{\bar{r}_t}(a) - \bar{\gamma}_{\bar{r}_t}(b) = \sum_{j=1}^{N^*_{\bar{r}_t}} \theta^*_j e^*_j + \theta^*_{N^*_{\bar{r}_t}+1} \left( t - \bar{\gamma}_{T^*_{N^*_{\bar{r}_t}}} \right)
\end{align*}
and $T^*_1(\partial)$ is on the physical time scale. We have
\begin{align*}
\bar{\gamma}_{\bar{r}_t}(a) - \bar{\gamma}_{\bar{r}_t}(b) = \theta^*_1\, t, \quad 0 \leq t < \bar{\gamma}_{T^*_1(\partial)}
\end{align*}
where
\begin{align*}
\mathbf{P}(N^*_{\bar{r}_t} = 0) = \mathbf{P} (\bar{\gamma}_{T^*_1(\partial)} > t) = \mathbf{P}(H_{\gamma^+_{\tau^+_0}} > t) = E_\alpha(- (\sigma/\eta) t^\alpha).
\end{align*}
In particular, $\mathbf{P}(N^*_{\bar{r}_t} \geq j) = \mathbf{P}(T^L_j < t)$ and $N^*_{\bar{r}_t}$ coincides with $N^L_t$. The residual process becomes
\begin{align*}
\theta^*_{N^L_t + 1} \left( t - \bar{\gamma}_{T^L_{N^L_t}} \right) = \theta^*_{N^L_t + 1} \left( t - T^L_{N^L_t} \right).
\end{align*}
Indeed the process increases linearly on the local time scale and the renewal epochs are given according to
\begin{align*}
T^L_{N^L_t} \leq t < T^L_{N^L_t + 1}.
\end{align*}
This rigorously confirms that
\begin{align*}
\bar{\gamma}_{\bar{r}_t}(a) - \bar{\gamma}_{\bar{r}_t}(b) = \sum_{j=1}^{N^L_t} \theta^*_j e^L_j + \theta^*_{N^L_t+1} \left( t - T^L_{N^L_t} \right).
\end{align*}
The structural framework established in items i)--v) provides a pathwise, path-by-path foundation for the local time dynamics. 
\end{proof}

Under the setting of Theorem \ref{thm:difUNO} and Theorem \ref{thm:difALPHA} we can prove that
\begin{align*}
\gamma^+_{\tilde{V}^{-1}_t}(a) - \gamma^+_{\tilde{V}^{-1}_t} (b) \stackrel{d}{=} \gamma^+_{r^+_{L_t}}(a) - \gamma^+_{r^+_{L_t}}(b)
\end{align*}
for every $t\geq 0$. From the representation \eqref{seriesTEL} with $Y_0=0$ and Theorem \ref{thmDiffLT} we also obtain that
\begin{align*}
\gamma^+_{r^+_t}(a) - \gamma^+_{r^+_t}(b) \stackrel{d}{=} \theta^*_0 \int_0^t (-1)^{N_s} ds
\end{align*}
for every $t\geq 0$ and $\theta^*_0$ uniformly distributed on $\{-1, +1\}$. Moreover, by comparing $Y^\alpha$ in Section \ref{sec:introTel} with \eqref{thmDifLTbar} we can write
\begin{align*}
\bar{\gamma}_{\bar{r}_t}(a) - \bar{\gamma}_{\bar{r}_t}(b) \stackrel{d}{=} \theta^*_0 \int_0^t (-1)^{N^L_s}\, ds
\end{align*}
for every $t\geq 0$ and $\theta^*_0$ uniformly distributed on $\{-1, +1\}$. Finally, we observe that, for $\alpha \in (0,1]$, 
\begin{align*}
\bar{\gamma}_{\bar{r}_t}(a) - \bar{\gamma}_{\bar{r}_t}(b) = \int_0^t \theta^*_{N^L_s + 1}\,  ds 
\end{align*}
holds pathwise, for all $t \geq 0$. This directly comes from  \eqref{thmDifLTbar}.

\subsubsection{Boundary analysis via Tanaka's formula and time-change}
\label{sec:further}

The stochastic evolution of $X^+$ can be written via the Skorokhod decomposition as
\begin{equation*}
X^+_t = X^+_0 + B_t + \gamma^+_t(a) - \gamma^+_t(b)
\end{equation*}
where $\{B_t\}_{t \geq 0}$ is a Brownian motion with variance $2t$. By considering the time-change $r^+_t$ into the Skorokhod decomposition, we obtain the boundary process
\begin{equation}
X^+_{r^+_t} = X^+_0 + B_{r^+_t} + \gamma^+_{r^+_t}(a) - \gamma^+_{r^+_t}(b).
\label{eq:time_changed_Skorokhod}
\end{equation}
The support of the time-changed process (boundary trace process) is strictly confined to the boundary,
\begin{equation*}
X^+_{r^+_t} \in \{a, b\} \quad \forall t \geq 0 \quad \textrm{ almost surely.}
\end{equation*}
In particular, the trace process $\theta^+_t := X^+_{r^+_t}$ evolves as a symmetric, continuous-time Markov chain. Since the time-change $r_{t}^{+}$ compresses the independent excursions of the Brownian motion into memoryless transitions, and since the accumulated boundary local time until the exit ($X^+$ hits the midpoint $0$) follows an exponential distribution, the trace process $\{\theta^+_t\}_{t \geq 0}$ inherits the strong Markov property. Thus, it evolves rigorously as a continuous-time Markov chain on the state space $\{a, b\}$ (see the sequence $\theta^*$ defined in Section \ref{sec:Path}). The time-changed process $B_{r^+_t}$ acts as the underlying stochastic source driving the jumps of the trace process. \\

For the sticky process $\bar{X}_t$, we similarly observe that
\begin{equation}
\bar{X}_t = \bar{X}_0 + B_{\bar{V}^{-1}_t} + \tilde{\gamma}_t(a) - \tilde{\gamma}_t(b),
\end{equation}  
and analogous arguments hold. Specifically, the time-changed noise $B_{V^{-1}_t}$ remains constant precisely when the process is trapped at the boundaries, reflecting the presence of positive holding times.

\subsection{$\Omega$ is a smooth domain: an illustrative example}
\label{sec:OmegaSmooth}

We focus on the disk, although our analysis can be adapted to balls in $\mathbb{R}^d$.   The sticky process $\bar{X}$ moves along the path of $X^+$ and $\bar{X}$ accumulates local time $\bar{\gamma}_t$ as $X^+$ accumulates local time $\gamma^+_t$. Rotational symmetry plays a crucial role. Thus, $\bar{X}$ needs more physical time to accumulate the same local time of $X^+$. Indeed, $\bar{V}_t = t + H_{\gamma^+_t}$ runs faster than $t$, the physical time of $X^+$.\\ 

Let us consider the balls $\Omega_1 \subset \Omega \subset \Omega_2$ with radii $r_1 < r <r_2$ centred at $x_0$  where $\Gamma_1= \partial \Omega_1$, $\Gamma= \partial \Omega$, $\Gamma_2= \partial \Omega_2$. We assume that $r_1 = r - \epsilon$ and $r_2=r + \epsilon$ where 
\begin{align*}
\epsilon = (\eta / \sigma).
\end{align*} 
In this case $\partial \Omega$ is a membrane with $\epsilon$-neighbourhood 
\begin{align*}
\Lambda_\epsilon = \{x \in \mathbb{R}^2\,:\, \operatorname{dist}(x, \partial \Omega) < \epsilon \} \quad \textrm{that is} \quad \Lambda_\epsilon = \Omega_2 \setminus \overline{\Omega}_1 
\end{align*}
and the process $X$ on $\Lambda_\epsilon$ is a Brownian motion with sticky boundary $\Gamma_1 \cup \Gamma_2$. While $\partial \Omega$ generally acts as a transmission or skew membrane, we assume symmetric reflection here for simplicity. Let us define
\begin{align*}
\tau^+_{\partial \Omega} :=\inf\{t\geq 0\,:\, X^+_t \in \partial \Omega \} .
\end{align*}
According to \cite[Remark 6.5]{EJP}, we are able to prove that 
\begin{align}
\gamma^+_{\tau^+_{\partial \Omega}}(\Gamma_i) \sim \operatorname{Exp}(\rho_i), \quad \textrm{under $\mathbf{P}_x$, $x \in \Gamma_i$}, \quad i=1,2
\label{expLocTim}
\end{align}
where 
\begin{align*}
\mathbf{E}_x[\gamma^+_{\tau^+_{\partial \Omega}}(\Gamma_1)] = (r-\epsilon) \ln (r/(r-\epsilon))=:1/\rho_1, \quad x \in \Gamma_1
\end{align*}
and 
\begin{align*}
\mathbf{E}_x[\gamma^+_{\tau^+_{\partial \Omega}}(\Gamma_2)] = (r+\epsilon) \ln ((r+\epsilon)/r)=:1/\rho_2, \quad x \in \Gamma_2.
\end{align*}
Observe that, due to the different curvatures, we get $\rho_1 \neq \rho_2$.\\

The purpose of this section is merely to provide an additional example; a detailed discussion is omitted. Our statement is supported by analysing the radial component $R_t = \|X^+_t - x_0\|$ of the underlying reflected process $X^+$ before the sticky time-change is applied. Formally, $R_t$ satisfies the reflected Bessel stochastic differential equation
\begin{align*}
dR_{t}=\frac{d-1}{R_{t}}dt+dB_t + d\ell _{1}(t)-d\ell _{2}(t)
\end{align*}
where $\ell_1$ and $\ell_2$ are the local times respectively associated to $\Gamma_1$ and $\Gamma_2$. \\

We obtain \eqref{expLocTim} by considering separately the $1$-potentials
\begin{align*}
u_i(x) = \mathbf{E}_x[ \exp(-\gamma^+_{\tau^+_{\partial \Omega}}(\Gamma_i)) ], \quad x \in \Lambda_\epsilon^i, \quad i=1,2
\end{align*}
where
\begin{align*}
\Lambda_\epsilon^1 = \Omega \setminus \bar{\Omega}_1, \quad \Lambda_\epsilon^2 = \Omega_2 \setminus \bar{\Omega} \quad \textrm{and} \quad \Lambda_\epsilon = \Lambda_\epsilon^1 \sqcup \partial \Omega \sqcup \Lambda_\epsilon^2.
\end{align*}
In particular, we show that
\begin{align*}
u_i(x) = \mathbf{E}_x[ \exp(-\gamma^+_{\tau^+_{\partial \Omega}}(\Gamma_i)) ] = \frac{1}{1 + 1/\rho_i}, \quad x \in \Gamma_i, \quad i=1,2
\end{align*}
and this proves \eqref{expLocTim}.\\

The functions $u_i$, $i=1,2$ are harmonic functions solving
\begin{align*}
\Delta u_i =0\, \textrm{ on } \Lambda_\epsilon^i,\quad  \partial_{\bf n} u_i = - u_i\, \textrm{ on }  \Gamma_i, \quad u_i=1\, \textrm{ on }  \partial \Omega
\end{align*}
and the outer normal derivative is defined on the external boundary $\Gamma_2$ and the internal boundary $\Gamma_1$, respectively. By exploiting the circular symmetry and assuming $x_0=0$, with a slight abuse of notation, we identify $u_i(x)$ with its radial component $u_i(z)$ for $|x|=z$, and consider the problems 
\begin{equation*}
\left\lbrace
\begin{array}{ll}
(u_1)^{\prime \prime}(z) + \frac{1}{z} (u_1)^\prime(z) = 0, & z \in (r-\epsilon, r),\\
-(u_1)^\prime(r-\epsilon) = -u_1(r-\epsilon),\\
u_1(r)=1.
\end{array}
\right . , \quad |x| = z \in [r-\epsilon, r]
\end{equation*}
and
\begin{equation*}
\left\lbrace
\begin{array}{ll}
(u_2)^{\prime \prime}(z) + \frac{1}{z} (u_2)^\prime(z) = 0, & z \in (r,r+\epsilon),\\
(u_2)^\prime(r+\epsilon) = -u_2(r+\epsilon),\\
u_2(r) = 1.
\end{array}
\right . , \quad |x|=z \in [r, r+\epsilon].
\end{equation*}
The solutions are given by
\begin{align*}
u_1(z) = 1 - \frac{(r-\epsilon) \ln (r/z)}{1 + (r-\epsilon) \ln(r/(r-\epsilon))}, \quad z \in [r-\epsilon, r]
\end{align*}
and
\begin{align*}
u_2(z) = 1 - \frac{(r+\epsilon) \ln(z/r)}{1 + (r+\epsilon) \ln((r+\epsilon)/r)}, \quad z \in [r, r+\epsilon]
\end{align*}
from which
\begin{align*}
u_1(r-\epsilon) = \frac{1}{1 + 1/\rho_1}, \quad u_2(r+\epsilon) = \frac{1}{1 + 1/\rho_2}
\end{align*}
as stated before.\\

We are therefore able to establish an explicit relationship for the difference 
\begin{align*}
\gamma^+_t(\Gamma_1) - \gamma^+_t(\Gamma_2) \quad \textrm{where} \quad \gamma^+_t(\Gamma_i) = \int_0^t \mathbf{1}_{\Gamma_i}(X^+_s) d\gamma^+_s, \quad  i=1,2
\end{align*}
and 
\begin{align*}
\gamma^+_t := \gamma^+_t(\Gamma_1 \cup \Gamma_2) = \gamma^+_t(\Gamma_1) + \gamma^+_t(\Gamma_2)
\end{align*}
is the total local time accumulated up to time $t$. Let us define the inverse $r^+_t$ for which $\gamma^+_{r^+_t}=t$. As in the previous section we can write
\begin{equation}
\gamma^+_{r^+_t}(\Gamma_1) - \gamma^+_{r^+_t}(\Gamma_2) = \int_0^t \tilde{\theta}_s \, ds
\end{equation}
where
\begin{align*}
\tilde{\theta}_t = \left( \mathbf{1}_{\Gamma_1}(X^+_{r^+_t}) - \mathbf{1}_{\Gamma_2}(X^+_{r^+_t}) \right) \in \{-1, +1\}, \quad t\geq 0
\end{align*}
provides the position of the trace process $X^+_{r^+_t}$. Radial symmetry guarantees the Markov property. Consequently, we expect that $\{\tilde{\theta}_t\}_{t \geq 0}$ is a continuous-time Markov chain with state space $\{-1, +1\}$. For example, given a renewal process $\{\tilde{N}_t\}_{t\geq 0}$ and the sequence of epochs $\{\tilde{T}_j\}_{j \in \mathbb{N}_0}$ we can write $\theta^*_j = \tilde{\theta}_{\tilde{T}_j}$ for $j \in \mathbb{N}$, thereby identifying an analogous construction to the one in Section \ref{sec:Path}. \\

Thus, $X^+$ accumulates an exponential amount of time $\gamma^+_{\tau^+_{\partial \Omega}}$ on a given boundary before leaving for the next boundary accumulation (that is, before reaching the middle membrane $\Gamma$ and moving to either $\Gamma_1$ or $\Gamma_2$). By definition, $\bar{X}$ accumulates a Mittag-Leffler amount of time  $H_{\gamma^+_{\tau^+_{\partial \Omega}}}$, allowing us to recover the same arguments as in the previous sections dealing with $\Lambda_\epsilon = [a, b]$. However, those results do not apply exactly here. Indeed, due to the curvatures of the internal and external boundaries, the exponential r.v.'s  are i.i.d. only for a given boundary. Nevertheless, they remain independent with respect to the total local time accumulation. In this case,  we likewise have
\begin{align*}
\gamma^+_{r^+_t}(\Gamma_1) + \gamma^+_{r^+_t}(\Gamma_2) = t
\end{align*}
which corresponds to the decomposition
\begin{align*}
\sum_{j=1}^{N_t} e^*_j + \big(t - \sum_{j=1}^{N_t} e^*_j\big) = t
\end{align*}
where $\{e^*_j\}_{j \in \mathbb{N}}$ are independent exponential r.v.'s whose distributions depend on $\theta^*$, which in turns characterizes the variables in terms of the appropriate parameter $\rho_i$, $i=1,2$.

\begin{figure}
\centering
\includegraphics[scale=1]{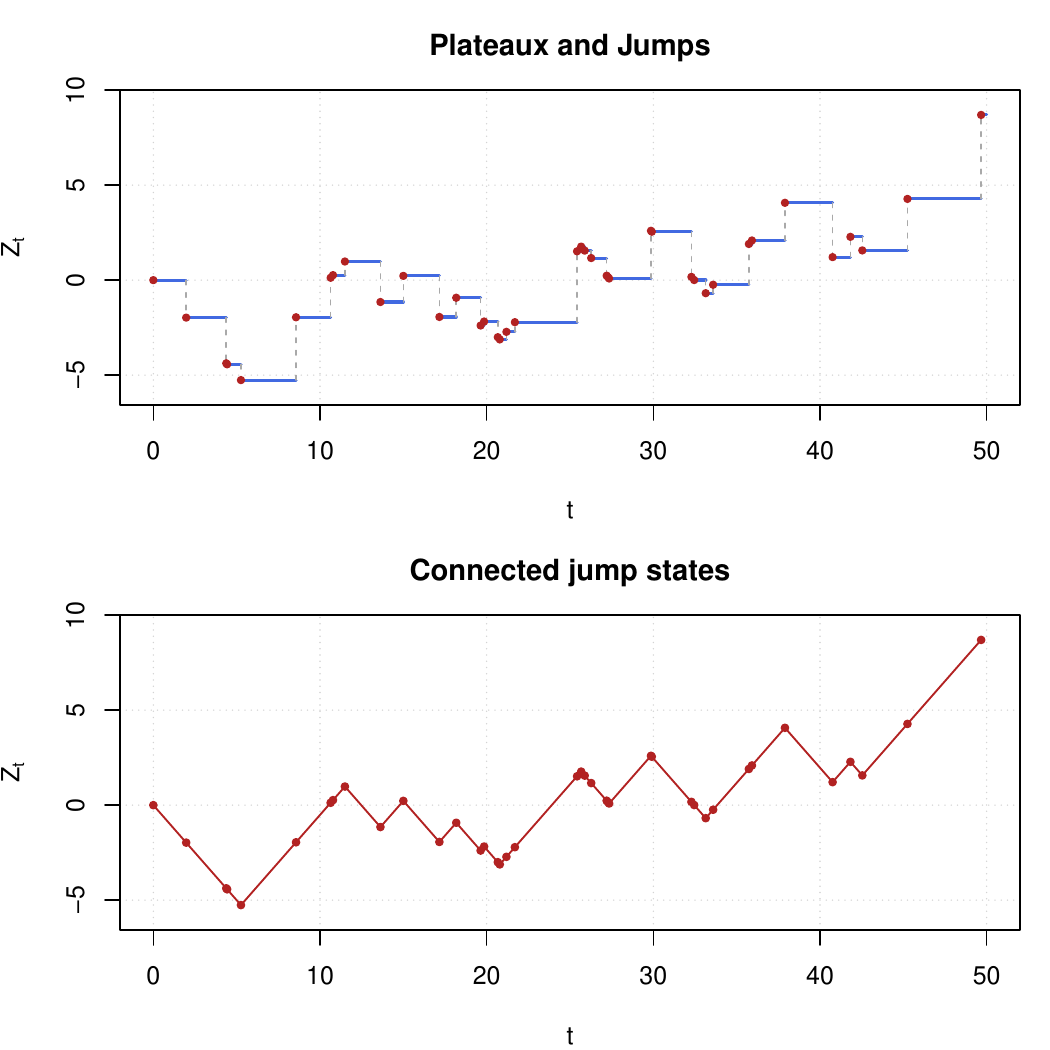} 
\caption{A realization of $Z$ for $\alpha=1$. The spatial jump size equals the time step, $\Delta x = \Delta t \sim Exp(\sigma/\eta)$. The residual process is linear.}
\label{fig:uno}
\end{figure}

\begin{figure}
\centering
\includegraphics[scale=1]{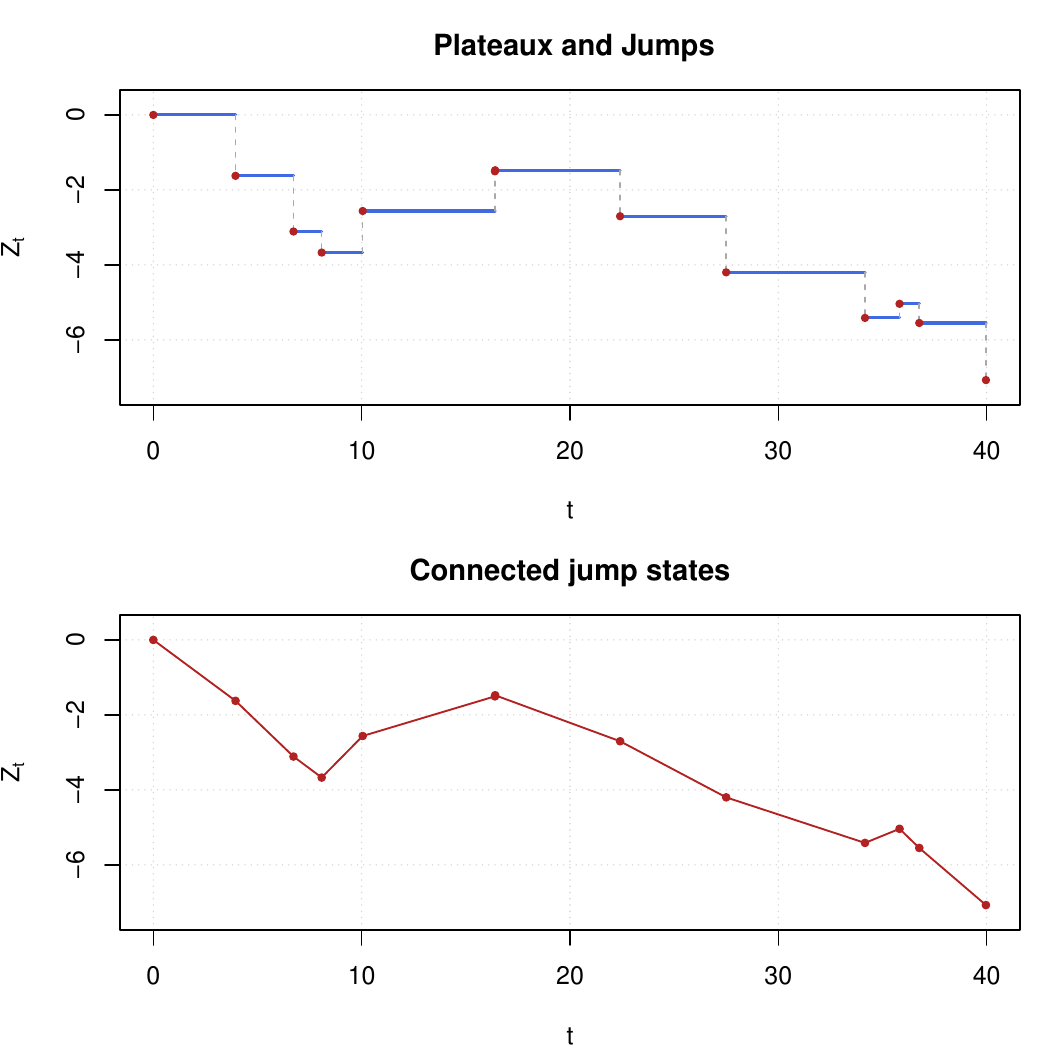} 
\caption{A realization of $Z$ for $\alpha=0.7$. The spatial jump size equals the time step, $\Delta x \neq \Delta t \sim H_{Exp(\sigma/\eta)}$. The residual process is linear.}
\label{fig:alpha}
\end{figure}

\begin{figure}
\centering
\includegraphics[scale=1]{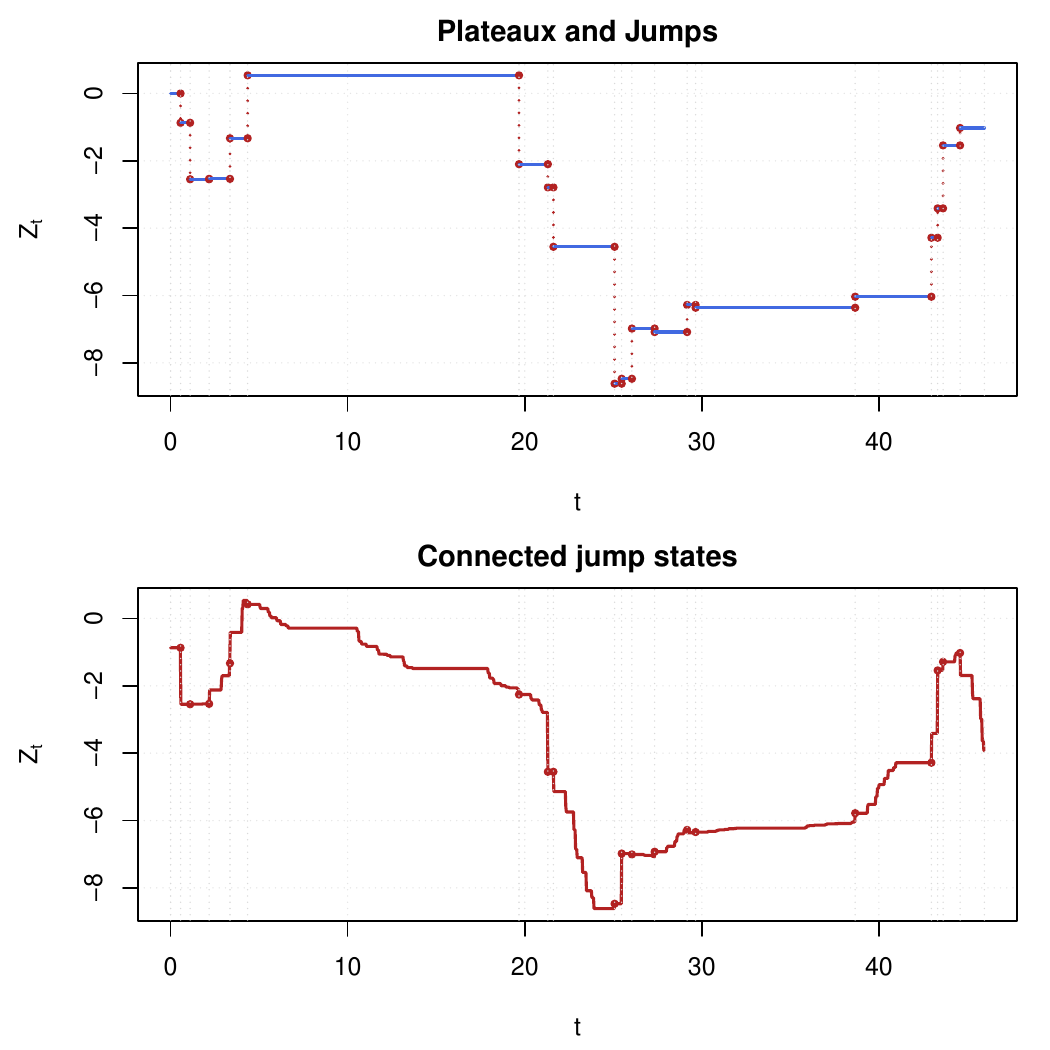} 
\caption{A realization of $Z$ for $\alpha=0.7$. The spatial jump size equals the time step, $\Delta x \neq \Delta t \sim H_{Exp(\sigma/\eta)}$. The residual process is non-linear. The vertical lines do not represent jumps; it should be emphasized that between renewal points, the residual process (that is $L$) can increase rapidly along a steep linear path. }
\label{fig:Ztel}
\end{figure}

\begin{figure}
\centering
\includegraphics[scale=.8]{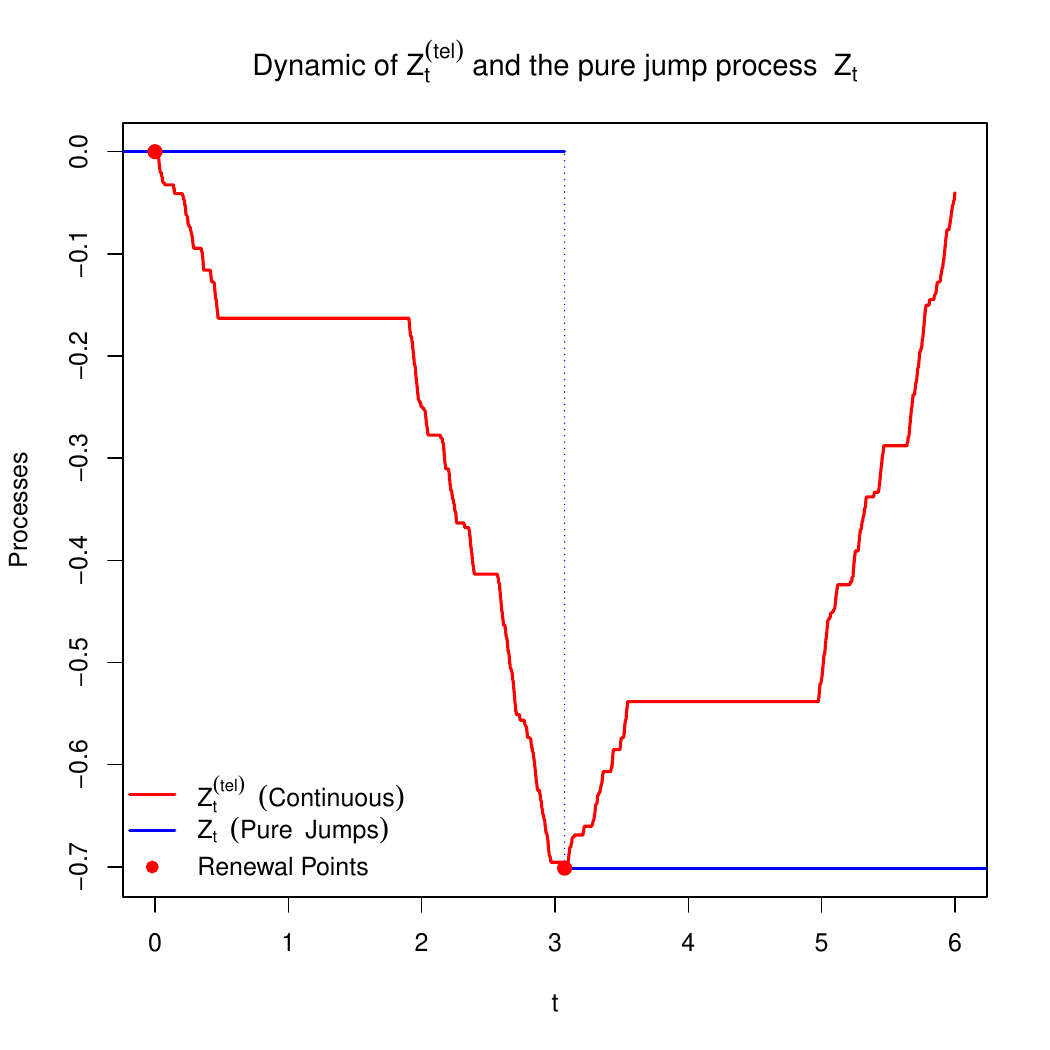} 
\caption{A realization of $Z$ and $Z^\mathtt{tel}$ for $\alpha=0.8$. The picture shows the path of $Z^\mathtt{tel}$ before and after a renewal point; instead of a straight line, the trajectory of $L$ connects two consecutive renewal points.}
\label{fig:ZtelPart}
\end{figure}

\end{document}